\documentclass[a4paper,11pt]{amsart}
\usepackage[utf8]{inputenc}
\usepackage[english]{babel}
\usepackage{frenchineq}
\usepackage{amsmath}  
\usepackage{amssymb}     
\usepackage{amsthm}  
\usepackage{bbm}    

\usepackage{bm}
\usepackage{hyperref}
\usepackage{accents}
\usepackage{mathrsfs}
\usepackage{mathtools}
\usepackage{fixmath}
\usepackage{fullpage}
\usepackage{paralist}
\usepackage{stmaryrd}

\usepackage{booktabs}
\usepackage{array}

\usepackage{pgf,tikz}
\usetikzlibrary{arrows}
\usetikzlibrary{patterns}
\usetikzlibrary[positioning]
\usetikzlibrary[fit]
\usetikzlibrary{shapes.geometric}
\usetikzlibrary{calc}

\usepackage{overpic}

\newtheorem{theorem}{Theorem}
\newtheorem{proposition}[theorem]{Proposition}
\newtheorem{corollary}[theorem]{Corollary}
\newtheorem{lemma}[theorem]{Lemma}
\newtheorem{assumption}[theorem]{Assumption}

\theoremstyle{remark}
\newtheorem{remark}[theorem]{Remark}
\newtheorem{ex}[theorem]{Example}

\setdefaultitem{$\triangleright$}{}{}{}

\usepackage{ifthen}
\usepackage[colorinlistoftodos,bordercolor=orange,backgroundcolor=orange!20,linecolor=orange,textsize=scriptsize]{todonotes}

\usepackage{xspace}

\DeclareSymbolFont{stmry}{U}{stmry}{m}{n}
\DeclareMathSymbol\mapsfromchar\mathrel{stmry}{"5B}

\newcommand{\ie}{{i.e.}\xspace} 
\newcommand{\eg}{{e.g.}\xspace} 
\newcommand{\etc}{{etc}\xspace}
\newcommand{\arxiv}[1]{\texttt{arXiv:#1}}

\DeclareMathOperator*{\argmax}{\arg\max}

\DeclareMathOperator*{\sign}{sign}

\newcommand{\hardyex}{{\bf LP}} 
\newcommand{\tropex}{{\rm LP}}

\newcommand{\troppath}{\mathcal{C}^{\text{\rm trop}}}
\newcommand{\troppathprimal}{\mathcal{C}^{\text{\rm trop}}_{\text{\rm primal}}}

\newcommand{\lang}{\mathcal{L}}

\newcommand{\theory}{\text{Th}}

\newcommand{\struct}{\mathfrak{R}}
\newcommand{\mystruct}{\bar{\R} ^ {\R}}

\DeclareMathAlphabet\mathbfcal{OMS}{cmsy}{b}{n}

\DeclareMathAlphabet{\mathbbold}{U}{bbold}{m}{n}

\newcommand{\R}[0]{\mathbb{R}}

\newcommand{\F}[0]{\mathbb {F}}                
\newcommand{\K}[0]{\mathbb {K}}

\newcommand{\trop}[1][]{
\ifthenelse{\equal{#1}{}}{ \mathbb{T} }{ \mathbb{T}(#1) }%
}

\newcommand{\transpose}[1]{#1^{\top}}

\newcommand{\tplus}{\oplus}          
\newcommand{\tsum}{\bigoplus}        
\newcommand{\ttimes}{\odot}       
\newcommand{\tdot}{\odot}

\newcommand{\tropP}{\mathcal{P}}  
\newcommand{\tropQ}{\mathcal{Q}}  
\newcommand{\tropR}{\mathcal{R}}

\DeclareMathOperator*{\val}{val}

\newcommand{\puiseuxP}[0]{\bm{\mathcal{P}}}   
\newcommand{\puiseuxQ}[0]{\bm{\mathcal{Q}}}   
\newcommand{\puiseuxS}[0]{\bm{\mathcal{S}}}

\newcommand{\x}[0]{\bm{x}}
\newcommand{\y}[0]{\bm{y}}

\newcommand{\puiseuxR}[0]{\bm{\mathcal{R}}}

\newcommand{\cpath}{\mathcal{C}}

\newcommand{\perturb}[2][]{
\ifthenelse{\equal{#1}{}}{ \tilde{\Ipert{#2}} }{ #2[#1] }%
}
\newcommand{\pert}[2][]{
\ifthenelse{\equal{#1}{}}{ \widetilde{\Ipert{#2}} }{ \Ipert{#2}[#1] }%
}

\newcommand{\Ipert}[1]{
\mathsf{#1}}

\newcommand{\norm}[1]{\|#1\|}
\newcommand{\scalar}[2]{\langle#1,#2\rangle}

\newcommand{\classicalpuiseux}{\R\{\!\{t\}\!\}}

\newcommand{\funk}{\delta_{\mathrm F}}
\newcommand{\hilbert}{d_{\mathrm H}}
\newcommand{\inverse}[1]{\jmath(#1)}

\newlength{\mytemplen}
\newcommand{\mybackup}[1]
{
  \settowidth{\mytemplen}{\(\displaystyle #1\)}
  \hskip-\mytemplen%
  \mkern-8mu
}

\title{Long and Winding Central Paths}

\author{Xavier Allamigeon
\and Pascal Benchimol
\and St\'ephane Gaubert
\and Michael Joswig} 

\address[Xavier Allamigeon, St{\'e}phane Gaubert]{INRIA and CMAP, \'Ecole Polytechnique, CNRS, 91128 Palaiseau Cedex France
\texttt{firstname.lastname@inria.fr}}
\address[Pascal Benchimol]{EDF Lab,
1 avenue du Général de Gaulle, BP 408,
92141, Clamart Cedex, France \texttt{pascal.benchimol@polytechnique.edu}}
\address[Michael Joswig]{
Institut f{\"u}r Mathematik,
 TU Berlin, 
 Str.\ des 17. Juni 136, 10623 Berlin, Germany
 \texttt{joswig@math.tu-berlin.de}
}

\thanks{The first and third authors are partially supported by the PGMO program of EDF and Fondation Math\'ematique Jacques Hadamard.}
\thanks{During this work, P.~Benchimol was affiliated with INRIA Saclay \^Ile-de-France and CMAP, \'Ecole Polytechnique,  CNRS UMR 7641. He was supported a PhD fellowship of DGA and \'Ecole Polytechnique.}
\thanks{M.~Joswig is partially supported by Einstein Foundation Berlin, DFG within the Priority Program 1489 and by a CNRS INSMI visiting professorship at CMAP, \'Ecole Polytechnique, UMR 7641 and IMJ, Universit\'e Pierre et Marie Curie, UMR 7586.}
\subjclass[2010]{90C51, 14T05}
\begin{document}

\begin{abstract}
  We disprove a continuous analogue of the Hirsch conjecture proposed by Deza, Terlaky and Zinchenko, by constructing a family of linear programs with
  $3r+4$ inequalities in dimension $2r+2$ where the central path has a total curvature in $\Omega(2^r)$. Our method is to tropicalize the central
  path in linear programming. The tropical central path is the piecewise-linear limit of the central paths of parameterized families of classical
  linear programs viewed through logarithmic glasses. The lower
bound for the classical curvature is obtained by developing a combinatorial
concept of a tropical angle.
\end{abstract}

\maketitle

\section{Introduction} 
\noindent
Since Karmarkar's seminal work~\cite{karmarkar1984new}, interior point methods have become indispensable in mathematical optimization. They provide
algorithms with a polynomial complexity in the bit model for linear programming.  Moreover, interior point methods are also useful for more general
convex optimization problems such as semi-definite programming.  Path-following interior point methods are driven to an optimal solution along a
trajectory called the \emph{central path}. Early on, Bayer and Lagarias
recognized that the central path is ``a fundamental mathematical
object underlying Karmarkar's algorithm and that the good convergence properties of Karmarkar's algorithm arise from good geometric properties of the set
of trajectories''~\cite[p.~500]{BayerLagarias89a}. 
Intuitively, a central path with high curvature should be harder to approximate with line segments, and
thus this suggests more iterations of the interior point methods. 
Dedieu and Shub conjectured that the total curvature of the central
path is linearly bounded in the dimension of the ambient space~\cite{dedieu2005newton}.  Dedieu, Malajovich and Shub showed that this property is valid in some average sense~\cite{dedieu2005curvature}. 
However, Deza, Terlaky and Zinchenko 
provided a counter example
by constructing a redundant Klee-Minty cube~\cite{DTZ08}.
This led them to state a continuous analogue of the Hirsch conjecture:
that the total curvature of the central
path is linearly bounded in the number of constraints.
The
purpose of this paper is to apply tools from tropical geometry to study the central paths. In this way, we disprove the conjecture of Deza, Terlaky
and Zinchenko. 

Tropical geometry can be seen as the (algebraic) geometry on the semiring $(\trop, \tplus, \ttimes)$ where the set $\trop = \R \cup \{ - \infty \}$ is
endowed with the operations $a \tplus b = \max(a,b)$ and $a \ttimes b = a +b$.  A tropical variety can be obtained as the limit at infinity of a
sequence of classical algebraic varieties depending on one real parameter $t$ and drawn on logarithmic paper, with $t$ as the logarithmic base. 
This process is known as Maslov's~\emph{dequantization}~\cite{litvinov2007maslov}, or Viro's method~\cite{viro2001dequantization}.  It can be traced
back to the work of Bergman~\cite{bergman1971logarithmic}.  In a way, dequantization yields a piece-wise linear image of classical algebraic
geometry.  Tropical geometry has a strong combinatorial flavor, and yet it retains a lot of information about the classical
objects~\cite{itenberg2009tropical,TropicalBook}.
 
The tropical semiring can also be thought of as the image of a non-archimedean field under its valuation map. This is the approach we adopt here.  The
non-archimedean fields typically used are the field of formal Puiseux series~\cite{kapranov,DevelinYu07, richter2005first}
or the field of generalized Puiseux series with real exponents~\cite{markwig2007field}, or larger fields of formal Hahn series~\cite{stacs}.  However,
since we are aiming at analytic results, matters of convergence play a key role, and this is why dealing with any kind of formal power series is not
suitable here.  Instead we take the viewpoint of Alessandrini~\cite{alessandrini2013} who suggested to study tropicalizations of real semi-algebraic
sets via a Hardy field, $\K$, of germs of real-valued functions.  The functions $f \in \K$ are definable in some o-minimal structure, which ensures a
tame topology.  In particular, the limit $ \lim_{ t \to \infty } \log ( f(t) ) / \log(t)$ always exists, and this defines a valuation
on~$\K$. Furthermore, this framework is flexible enough to include all power functions into $\K$; this makes the valuation map surjective
onto $\R \cup \{ - \infty \}$.

We consider linear programs defined on the Hardy field $\K$. As $\K$ is an ordered field, the basic results of linear programming (Farkas' lemma,
strong duality, \etc) still hold true on $\K$.  Also, since $\K$ is real closed, the central path of a linear program is well-defined.  The
elements of~$\K$ are real-valued functions. As a result, a linear program over $\K$ encodes a family of linear programs over $\R$, and the central path on $\K$
describes the central paths of this family. The tropical central path is then defined  as the image under the valuation map.  Thus, the tropical
central path is the logarithmic limit of a family of classical central paths.  We establish that this convergence is uniform.

The tropical central path has a purely geometric characterization.  Applying the valuation map to the feasible region
yields a tropical polyhedron $\tropP$.  We show that the tropical analytic center is the greatest element of this
tropical polyhedron, the tropical equivalent of a barycenter. Thus, the tropical analytic center does not depend on the
external representation of the feasible set.  Similarly, any point on the tropical central path is the tropical
barycenter of the set obtained by intersecting $\tropP$ with a tropical sublevel set induced by the objective
function. This is in stark contrast with the classical case, where the central path depends on the halfspace description
of the feasible set.  For an example, see~\cite{DTZ08}.

A maybe surprising feature is that the tropical central path can degenerate to a path taken by the tropical simplex
method introduced in \cite{tropical+simplex,stacs}.  We can even provide a quite general sufficient condition for this kind of degeneration.
Consequently, the tropical central path may have the same worst-case behavior as the simplex method.  

A main contribution of this paper is the study of the total curvature of the real central paths arising from lifting
tropical linear programs to the Hardy field $\K$.  This leads to a family of linear programs with $3r+4$ inequalities in
dimension $2r+2$ where the central path has a total curvature in $\Omega(2^r)$.  In fact, the tropical central path
shows a self-similar pattern, which has a staircase-like shape with $\Omega(2^r)$ steps.  This provides a counter
example to the continuous analogue of the Hirsch conjecture.  Our family of linear programs is gotten by lifting
tropical linear programs which come from a construction of Bezem, Nieuwenhuis and
Rodr{\'{\i}}guez-Carbonell~\cite{BezemNieuwenhuisRodriguez08}. Their goal was to show that an algorithm of Butkovi\v{c}
and Zimmermann~\cite{Butkovic2006} has exponential running time. In order to estimate the curvature in this counter
example, we introduce a notion of tropical angle. This allows us to define a tropical analogue of the total curvature
which provides a lower bound for the classical total curvature. This notion of tropical angle and related metric
properties of real tropical varieties could be of independent interest.

\subsection*{Related Work.}
The possible simplex-like behavior of interior point methods was already observed by Megiddo and
Shub~\cite{megiddo1989boundary} and by Powell~\cite{powell1993number}. 
The redundant Klee-Minty cube of~\cite{DTZ08} and the ``snake'' in~\cite{deza2008polytopes} are instances which show that that the total curvature of the central path can be in $\Omega(m)$ for a polytope described by $m$ inequalities.  
Gilbert, Gonzaga and Karas~\cite{CharlesGilbert2004} also
exhibited ill-behaved central paths.  They showed that the central path can have a ``zig-zag'' shape with infinitely
many turns, on a problem defined in $\R^2$ by non-linear but convex functions.  

The central path has been studied by Dedieu, Malajovich and Shub~\cite{dedieu2005curvature} via the multihomogeneous B\'ezout
Theorem and by De Loera, Sturmfels and Vinzant~\cite{de2010central} using matroid theory.  These two papers provide an
upper bound of $O(n)$ on the total curvature averaged over all regions of an arrangement of hyperplanes in
dimension~$n$.  

In terms of iteration-complexity of
interior point methods, several worst-case results have been
proposed~\cite{anstreicher1991performance,kaliski1991convergence,ji1994complexity,powell1993number,todd1996lower,Bertsimas1997}.
In particular, Stoer and Zhao~\cite{zhao1993estimating} showed the iteration-complexity of a certain class of
path-following methods is governed by an integral along the central path. This quantity, called \emph{Sonnevend's
  curvature}, was introduced in~\cite{sonnevend1991complexity}. The tight relationship between the total Sonnevend
curvature and the iteration complexity of interior points methods have been extended to semi-definite and symmetric cone
programs~\cite{kakihara2013information}. 
Note that Sonnevend's curvature is different from the \emph{geometric} curvature we study in this paper. To the best of
our knowledge, there is no explicit relation between the geometric curvature and the iteration-complexity of
interior point methods. However, these two notions of curvature share similar properties. For instance, the total
geometric curvature and the total Sonnevend curvature are both maximal when the number of inequalities is twice the
dimension~\cite{deza2008polytopes,mut2013analogue}.  Further, on the redundant Klee-Minty cube of Deza, Terlaky and
Zinchenko, both the total geometric curvature and the total Sonnevend curvature are large~\cite{mut2013tight,DTZ08}.

\section{Preliminaries}

\subsection{The Hardy field} In this section, we recall the elements and results of model theory which are needed to define the Hardy field. We refer the reader to~\cite{MarkerDavid} for more background.

\subsubsection*{Languages and first-order formulae}
A \emph{language} $\lang = (\mathcal{R}, \mathcal{F}, \mathcal{C})$ consists of a set $\mathcal{R}$ of relations, a set $\mathcal{F}$ of functions,
and a set $\mathcal{C}$ of constants.  Each relation $R$ is equipped with an \emph{arity}, $n_R$, which is a positive integer.  Similarly, each
function $F$ also has an arity, denoted as $n_F$. For example, the language of ordered rings is $\lang_{\text{or}} = (\{< \}, \{ +, -, \cdot  \}, \{0,1
\}) $, where the order relation $<$ and the arithmetic functions $+,-,\cdot$ have arity two.

We shall now describe the (first-order) formulae of a language $\lang$.
An $\lang$-\emph{term} is either
\begin{itemize}
\item a variable $v_i$, for some $i \geq 1$, or
\item a constant $c \in \mathcal{C}$, or
\item $F(t_1, \dots, t_{n_F})$ where $F \in \mathcal{F}$ is a function, and $t_1, \dots t_{n_F}$ are $\lang$-terms.
\end{itemize}

An $\lang$-\emph{formula} is then defined inductively as follows:
\begin{itemize}
\item if $t_1$ and $t_2$ are $\lang$-terms, then $t_1 = t_2$ is an $\lang$-formula;
\item if $R \in \mathcal{R}$ is a relation, and $t_1, \dots t_{n_R}$ are terms, then $R(t_1, \dots t_{n_R})$ is an $\lang$-formula;
\item if $\phi$ and $\psi$ are $\lang$-formula, then $(\neg \phi)$, $(\phi \wedge \psi)$ and $(\phi \vee \psi)$ are $\lang$-formulae;
\item if $\phi$ is an $\lang$-formula and $v_i$ is a variable, then $\exists v_i \phi$ and $\forall v_i \phi$ are $\lang$-formulae.
\end{itemize}

A variable $v_i$ which occurs in a formula $\phi$ without being modified by a quantifier $\exists$ or $\forall$ is said to be \emph{free}. We shall
emphasize the free variables $v_{i_1}, \dots, v_{i_k}$ of a formula $\phi$ by writing $\phi(v_{i_1}, \dots, v_{i_k})$. A formula without free variable
is called a \emph{sentence}.

\subsubsection*{Structures}
Let $\lang = (\mathcal{R}, \mathcal{F}, \mathcal{C})$ be a language. An $\lang$-structure $\mathfrak M$ consists of a non-empty set $M$ (called the
\emph{domain} of $\mathfrak M$) together with an \emph{interpretation} of the symbols of $\lang$ in $M$. A relation $R \in \mathcal R$ is interpreted
by a subset $S_R \subset M^{n_R}$, where a tuple $(x_1, \dots, x_{n_R})$ satisfies the relation $R$ if $(x_1, \dots, x_{n_R}) \in S_R$. A function $F
\in \mathcal F$ is interpreted by a map $M^{n_f} \to M$, and the interpretation of a constant $c \in \mathcal C$ is an element of $M$.

The interpretation of the language $\lang$ induces an interpretation of the formulae of $\lang$ in the structure $\mathfrak M$. Every formula
$\phi(v_{i_1}, \dots, v_{i_k})$ defines a Boolean function $\phi^{\mathfrak M}$ on $M^k$.  If $\phi^{\mathfrak M}$ is true at $a \in M^k$, we
write $\mathfrak M \models \phi(a)$.  
In particular, if $\phi$ is a sentence in $\lang$, the function $\phi^{\mathfrak M}$ is constant. Thus a
sentence $\phi$ defines a statement on $\mathfrak M$ which is either true or false.  The set of sentences that are true on $\mathfrak M$ is called the
\emph{full theory} of $\mathfrak M$; it is denoted by $\theory(\mathfrak{M})$.  An arbitrary $\lang$-structure $\mathfrak{N}$ is a \emph{model} of the
theory $\theory(\mathfrak{M})$ if $\mathfrak{N} \models \phi$ for all $\phi \in \theory(\mathfrak{M})$.

A set $A \subset M^k$ is \emph{definable} (in $\mathfrak M$) if there exists an $\lang$-formula $\phi(v_1, \dots v_k, w_1, \dots, w_\ell)$ and an element $b \in
M^\ell$ such that $A = \{ a \in M^k \mid \mathfrak M \models \phi( a, b) \}$. Given a definable set $A \subset M^k$, a map $F: A \to M^\ell$ is \emph{definable}
if its graph $\{ (a,F(a) ) \mid a \in A \} \subset M^{k+\ell}$ is a definable set.

An \emph{expansion} $\lang'$ of a language $\lang$ is obtained by adding some new relations, functions and constants to $\lang$. We define an
\emph{expansion} of an $\lang$-structure $\mathfrak M$ to be an $\lang'$-structure $\mathfrak M'$ such that:  $\lang'$ is an expansion of $\lang$, 
$\mathfrak M$ and $\mathfrak M'$ have the same domain
and the interpretation of the language $\lang$ in $\mathfrak M$ coincides with the one in $\mathfrak M'$.

\subsubsection*{O-minimal structures and Hardy fields}

The  $\lang_{\text{or}}$-structure of the ordered field of real numbers is denoted by  $\bar{\R}= (\R, \{<\}, \{+,-, \cdot\}, \{0, 1\})$.
Throughout the following, $\lang$ will denote an expansion of the language $\lang_{\text{or}}$. Furthermore, $\struct$ will be a $\lang$-structure with domain $\R$ (thus an expansion of $\bar{\R}$) that is also  \emph{o-minimal}, which means that any subset of $\R$ definable in $\struct$  is a finite union of points and intervals with endpoints in $\R\cup\{ - \infty, + \infty \}$.
Under the o-minimality requirement, definable sets and maps are ``well-behaved''. For example, the set $\{(x, \sin(1/x)) \mid x > 0 \}$ is not definable in any o-minimal structure. We refer the reader to~\cite{van1998tame} or~\cite{Coste1999} for more information.

We say that two definable functions $f,g : \R \to \R$ are equivalent, and we write $f \sim g$, if $f(t) =g(t)$ \emph{ultimately}, \ie\ for all $t$ large enough.
The \emph{germ} $\bm f$ of a definable function $f$ is the equivalence class of $f$ for the relation $\sim$. By abuse of notation, $\bm f$ shall also denote a representative of the germ $\bm f$.

Let $H(\struct) := \{
\bm f \mid f \text{ definable in } \struct \}$ the set of germs of functions definable in $\struct$. Each function symbol $F \in \mathcal{F}$ has a natural interpretation in $H(\struct)$, by defining $F(\bm f_1, \dots, \bm f_{n_F})$ as the germ of the definable function $t \mapsto  F(\bm f_1(t), \dots, \bm f_{n_F}(t))$. Besides, the set $\R$ is embedded into $H (\struct)$ by identifying each element $a \in \R$ with the constant function with value $a$. This provides an interpretation of the constant symbols of $\mathcal{L}$ in $H(\struct)$. Finally, given a relation $R$ of the language $\lang$ and $\bm f_1, \dots , \bm f_{n_R} \in H(\struct)$, the set $\{t \mid \struct \models R(\bm f_1(t), \dots , \bm f_{n_R}(t)) \}$ is definable, and thus consists in a finite union of points and intervals. Hence, $R(\bm f_1(t), \dots , \bm f_{n_R}(t))$ is either ultimately true or ultimately false. This provides an interpretation of $R$ over $H(\struct)$. 

Consequently, $H(\struct)$ has a natural $\lang$-structure, which we
denote by $\mathfrak H (\struct)$.
It follows from \cite[Prop.~5.9]{Coste1999} that $\mathfrak H (\struct)$ and $\struct$ have the same full theory; see
also \cite[Lemma~2.2.64]{foster2010power}. 
In other words, the following holds.
\begin{proposition} \label{prop:model_complete}
For any $\lang$-sentence $\phi$, we have $\struct \models \phi$ if and only if $\mathfrak H (\struct) \models \phi$. 
\end{proposition}
As an expansion of $\bar{\R}$, the structure $\struct$ satisfies the axioms of the theory of real closed fields.  An
ordered field $\F$ is called \emph{real closed} if every positive element of $\F$ has a square root in $\F$ and every
odd degree polynomial with coefficients in $\F$ has at least one zero in $\F$.  By \emph{Tarski's Principle} a real
closed field has the same first-order properties as the reals. We deduce from Proposition~\ref{prop:model_complete}:
\begin{corollary}\label{th:hardy_real_closed}
The set $H(\struct)$ is a real closed field.	
\end{corollary}
We will refer to $H(\struct)$ as the \emph{Hardy field} of structure~$\struct$. In particular, $H(\struct)$ is an ordered field, and it carries a natural topology induced by the ordering. The standard topology on $\R$ coincides with the subspace topology induced from $H(\struct)$. 

A structure $\struct$ is \emph{polynomially bounded} if for any definable function $f : \R \to \R$, there exists a
natural number $n$ such that ultimately $|f(t)| \leq t^n$. Miller proved \cite{Miller1994} that in an o-minimal and
polynomially bounded expansion of $\bar{\R}$, if a definable function $f$ is not ultimately zero, then there exists an
exponent $r \in \R$ and a non-zero coefficient $c \in \R$ such that
\begin{equation}\label{eq:exponent}
  \lim_{t \rightarrow + \infty } \frac{f(t)}{t^r} = c \, .
\end{equation}
The set of such exponents $r$ forms a subfield of $\R$, called the \emph{field of exponents} of the structure~$\struct$.

\subsubsection*{Logarithmic limits of definable functions}

In the following, we will use the structure $\mystruct$ which expands $\bar \R$ by adding the family of power functions $(f_r) _{r \in \R}$,
where $f_r$ maps a positive number $t$ to $t^r$, and any non-positive number to $0$. The structure $\bar \R ^{ \R}$ is o-minimal, polynomially bounded and
its field of exponents is $\R$; see \cite{Miller1994expansions, miller2012basics}.
For the sake of readability, we shall abbreviate $H (\mystruct)$ by $\K$. We also use the notation $t^r$ as a shorthand for the germ of the power function $f_r$.
The \emph{valuation} maps any $\bm f \in \K$ to:
\[
\val(\bm f) := \lim_{t \to + \infty} \log_t  | \bm f(t) | \, ,
\]
where $\log_t(x) = \log(x) / \log (t)$.  By~\eqref{eq:exponent} the limit above is well-defined. Notice that $\val (\bm
f) = - \infty$ if $\bm f(t)$ ultimately vanishes. Since $\mystruct$ has $\R$ as its field of exponents, the valuation is a surjective map from $\K$ to $\R \cup \{ - \infty
\}$.  For $\bm f ,\bm g \geq 0$ we have:
\begin{align}\label{e-morphism}
 \val (\bm f + \bm g) = \max ( \val( \bm f), \val( \bm g)) , \quad \val (\bm f \bm g ) = \val (\bm f) + \val (\bm g) \, .
\end{align}
Moreover, if $\bm f \geq \bm g \geq 0$, then $\val(\bm f) \geq \val( \bm g)$.
Hence, the valuation map is an order-preserving homomorphism from the semiring of germs of definable functions which are ultimately non-negative to the tropical semiring. 
In the sequel, the valuation map  will be also be applied, being understood entrywise, to vectors or matrices with entries in the Hardy field $\K$. 

\subsubsection*{Comparison with other fields of generalized power series}\label{subsec-fields}
Several non-archimedean fields, which differ from the Hardy field $\K=H (\mystruct)$, have been used in the tropical literature, and it may be useful
to review alternative choices.  It is common to consider the field $\classicalpuiseux$ of formal Puiseux series with real
coefficients~\cite{kapranov,DevelinYu07, richter2005first}.  Technically, however, it is inconvenient that the valuation map from $\classicalpuiseux$ to $\trop$ is not surjective, as classical Puiseux series
have rational exponents.  This can be remedied by using the larger field of Hahn series with real coefficients, denoted by
$\llbracket\R^{\R,\leq}\rrbracket$ in~\cite{ribenboim}. An element of this field is a formal series of the form
\begin{align}
f = \sum_{\alpha \in \R} a_\alpha t^\alpha
\label{e-series}
\end{align}
with $a_\alpha\in \R$, such that the support $\{\alpha\in \R\mid a_\alpha \neq 0\}$ is well ordered. This field is known to be real closed~\cite{ribenboim}. 
An alternative to $\llbracket\R^{\R,\leq}\rrbracket$ is the subfield of generalized
formal Puiseux series, considered by Markwig~\cite{markwig2007field},
which consists of those series $f$ such that the support
$\{\alpha \in \R\mid a_\alpha \neq 0\}$ is either finite or
has $+\infty$ as the only accumulation point. It follows
from~\cite{markwig2007field} that this subfield is also real
closed. 

Our previous work~\cite{tropical+simplex,stacs} was developed using formal Hahn series. However, in the present
application, we need to work with fields of functions, thinking of $t$ as a deformation parameter.  We note that the
subfield of Markwig's field, consisting of the generalized Puiseux series that are absolutely convergent in a punctured
complex disc $0<|t|<r$, for some $r>0$,
actually coincides with the field $\mathbb{D}$ of generalized Dirichlet series originally
considered by Hardy and Riesz~\cite{hardy}, already used in the tropical setting in~\cite{ABG96}. Classical Dirichlet
series can be written as $\sum_k a_k k^s$, they are obtained from~\eqref{e-series} by substituting $t=\exp(s)$, with
$\alpha_k =\log k$.  It follows from results of van den Dries and Speissegger~\cite{Dries1998} that the field
$\mathbb{D}$ is real closed, and that the subfield $\llbracket\R^{\R,\leq}\rrbracket_{\text{cvg}}$ of
$\llbracket\R^{\R,\leq}\rrbracket$ consisting of series that are absolutely convergent in a punctured disk is also real
closed. Actually, the elements of the latter field can be identified to the germs of functions of one variable that are
definable in the o-minimal structure $\R_{\text{an},*}$, which is the expansion of the reals by restricted analytic
functions and convergent generalized power series \cite{Dries1998}.  However, it is enough to work here with the
o-minimal structure $\mystruct$, which is smaller than $\R_{\text{an},*}$.

\subsection{Tropicalization of linear programs}\label{subsec:tropicalization}
The tropical addition $\tplus$ extends to vectors and matrices by applying it coordinatewise.  Similarly, the tropical multiplication $\tdot$ gives rise to a
tropical multiplication of a scalar with a vector and, combined with $\tplus$, also to a tropical matrix multiplication.  A \emph{tropical halfspace}
of $\trop^n$ is the set of points $x \in \trop^n$ which satisfy one tropical linear inequality,
\[
\max( \alpha^+_1 + x_1, \dots, \alpha^+_n + x_n, \beta^- ) \leq \max(\alpha^-_1 + x_1, \dots, \alpha^-_n + x_n, \beta^+ ) \,,
\]
where $\alpha^+, \alpha^- \in \trop^{n}$ and $\beta^+, \beta^- \in \trop$.  
This formulation is the proper tropical analogue of the classical $\transpose{\bm{\alpha}} \bm{x} \leq \bm{\beta}$.
A \emph{tropical polyhedron} is the intersection
\[
\tropP = \{ x \in \trop^n \mid A^+ \tdot x \tplus b^- \leq A^- \tdot x \tplus b^+ \}
\] 
of finitely many tropical halfspaces, where $A^+, A^- \in \trop^{m \times n}$ and $b^+, b^- \in \trop^m$.
The tropical semiring is equipped with the order topology, which determines a product topology on $\trop^n$. Note that tropical halfspaces, and so, tropical polyhedra, are closed in this topology.

An analogue of the Minkowski--Weyl Theorem~\cite{GaubertKatz2011minimal} allows one to represent a tropical polyhedron internally in terms of extreme points and rays, meaning that there exist two finite collections of vectors
$v^1,\dots, v^r\in \trop^n$ and
$w^1, \dots,  w^s\in \trop^n$ such that
$\tropP$ can be written as the set of points
of the form 
\begin{align}
\tsum_{i = 1}^r \alpha_i \tdot v^i \; \tplus \;  \tsum_{j = 1}^s\beta_j\tdot w^j 
\label{e-minkowskiweyltropical}
\end{align}
where $\alpha_i,\beta_j\in \trop$ and $\tsum_{i\in[r]}\alpha_i$ is
equal to the tropical unit, \ie the real number $0$.  Here and below we use the common abbreviation $[r]:=\{1,2,\dots,r\}$.  We shall
say that $\tropP$ is \emph{generated} by 
$v^1,\dots, v^r$ and $w^1, \dots,  w^s$. 
Note that the ``tropical polytopes'' considered by 
Develin and Sturmfels~\cite{develin2004} are obtained
by omitting the $v^i$ terms
and by requiring the $w^j$ to have finite coordinates
in the representation~\eqref{e-minkowskiweyltropical}.

We denote by $\K_+$ the set of nonnegative elements of $\K$, and by $\K_+^n$
the \emph{positive orthant} of $\K^n$. The following fact was
already noted by Develin and Yu \cite[Proposition~2.1]{DevelinYu07} for ``tropical polytopes'' in the
sense of~\cite{develin2004}.
\begin{proposition}
\label{prop-direct}
The image under the valuation map of any polyhedron $\puiseuxP$ included in the positive orthant $\K_+^n$ is a tropical polyhedron of $\trop^n$. 
\end{proposition}
\begin{proof}
The Minkowski--Weyl theorem is valid for a polyhedron in any ordered
field. Hence, there exists two finite collections of vectors
$\bm v^1,\dots,\bm v^r\in \K_+^n$ and
$\bm w^1, \dots, \bm w^s\in \K_+^n$ such that $\puiseuxP$
is precisely the set of combinations of the following form
\begin{align}
\bm x = \sum_{i = 1}^r \bm \alpha_i \bm v^i  + \sum_{j = 1}^s\bm \beta_j \bm w^j
\label{e-minkowskiweyl}
\end{align}
where $\bm\alpha_i,\bm\beta_j\in \K_+$ and $\sum_{i\in[r]}\bm\alpha_i =1$.
Since the valuation is a homomorphism from $\K_+$ to $\trop$
(see~\eqref{e-morphism}), $\val(\puiseuxP)$ is included
in the tropical polyhedron $\tropP$ generated
by the vectors $v^1:=\val \bm v^1,\dots,v^{r}:=\val \bm v^r$ and
$w^1:=\val \bm w^1, \dots, w^s:=\val \bm w^s$. Conversely,
any point in $\tropP$ of the form~\eqref{e-minkowskiweyltropical}
is the image under the valuation map of 
\[
 \sum_{i = 1}^r \frac{1}{\bm Z} t^{\alpha_i} \bm v^i  + \sum_{j = 1}^s t^{\beta_j} \bm w^j
\]
where $\bm Z= \sum_{i = 1}^r t^{\alpha_i}$ is such that $\val \bm Z=0$.
\end{proof}

Conversely, each tropical
polyhedron arises as the image under the valuation map of a polyhedron included in $\K_+^n$. A slightly stronger statement can be established for tropical linear programs. A \emph{tropical linear program} asks to
minimize a tropical linear function $x \mapsto \transpose{c} \tdot x$ on a tropical polyhedron, where $c \in \trop^n$.   
The Main Lemma of tropical linear programming
establishes that
each tropical linear program arises as the image of some Hardy linear program under the valuation map. This means that the tropical feasible set is the image of the Hardy feasible set, and that optimal
  solutions are sent to optimal solutions,
see~{\cite[Proposition~7]{tropical+simplex}}.
However, a tropical linear program  may have optimal solutions which do not arise as images of Hardy optimal solutions under the valuation map.
Note that the results of~\cite{tropical+simplex} were proved
when the coefficients of the linear program belong to
the field of formal generalized Puiseux series~\cite{markwig2007field}. 
The same arguments apply to other non-archimedean real closed fields with residue field $\R$
that are sent surjectively to $\trop$ by the valuation,
including the Hardy field $\K=H (\mystruct)$; see Section~\ref{subsec-fields}.

Under tropical genericity conditions, we can directly obtain a halfspace description of $\val(\puiseuxP)$ from a halfspace description of $\puiseuxP$.
Since this is relevant for this paper, we will now describe this in more detail.  To ease the connection with the tropical description, assume that
$\puiseuxP$ is given as the set of $\bm x \in \K^n$ satisfying linear inequalities of the form $\bm A \bm x \leq \bm b$.
We additionally assume
that $\puiseuxP$ is contained in the positive orthant of $\K^n$.
 
We say that the \emph{tropicalization} of a matrix $\bm M \in \K^{n \times n}$ is \emph{sign non-singular} if $\det (\bm M) \neq 0$ and all the terms
$\sign(\sigma)
\prod_{i\in[n]} \bm M_{i \sigma(i)}$ with maximal valuation among those
arising in the expansion
of 
\[
\det \bm M = \sum_{\sigma \in S_n} \sign(\sigma)\prod_{i\in[n]} {\bm M}_{i\sigma(i)}
\]
share the same sign.
The tropicalization of a rectangular matrix $\bm W \in \K^{m \times n}$ is said
to be \emph{sign generic} if for every square submatrix $\bm M$ of $\bm W$, either the tropicalization of $\bm M$ is sign non-singular,  or for every permutation $\sigma$, the term  $\prod_{i\in[n]} {\bm M}_{i\sigma(i)}$ vanishes.

For any matrix $\bm M = (\bm M_{ij}) \in \K^{m \times n}$, we denote by $\bm M^+ = (\bm M^+_{ij})$ and $\bm M^- = (\bm
M^-_{ij})$ its positive and negative parts, \ie\ $\bm M^+_{ij} = \max (\bm M_{ij}, 0)$ and $\bm M^-_{ij} = \min(\bm
M_{ij}, 0)$.  Furthermore, $\bm M_I$ will denote the submatrix of $\bm M$ formed by the rows indexed by $I \subset [m]$.
\begin{theorem}[{\cite[Theorem~15 and Corollary~16]{tropical+simplex}}]\label{thm:val_inter_commutes}
  Suppose that the Hardy polyhedron $\puiseuxP = \{ \bm x \in \K^n \mid \bm A \bm x \leq \bm b \}$ is included in the
  positive orthant of $\K^n$ and that the tropicalization of $(\bm A, \bm b)$ is sign generic.  Then,
  \[
  \val ( \puiseuxP ) = \{ x \in \trop^n \mid A^+ \tdot x \tplus b^- \leq A^- \tdot x \tplus b^+ \} \, ,
  \]
  where $(A^+ \ b^+) = \val( \bm A^+ \bm b^+)$ and $(A^- \ b^-) = \val (\bm A^- \ \bm b^-)$.
  Moreover, for any $I \subset [m]$, we have:
  \[
  \val \left ( \{ \bm x \in \puiseuxP \mid \bm A_I \bm x = \bm b_I \} \right ) = \{ x \in \val(\puiseuxP) \mid A^+_I
  \tdot x \tplus b^-_I = A^-_I \tdot x \tplus b^+_I \} \, .
  \]
\end{theorem}

\subsection{The central path of primal-dual pair of linear programs}\label{subsec:central_path_prelim}
Let us now recall the definition of the central path arising in linear programming.  The reader is referred to the text
book of Roos, Terlaky and Vial~\cite{roos} for more information.  We consider a primal linear program of the following
form
\begin{align}\tag*{$\text{LP}(A,b,c)$}
& \begin{array}{r@{\quad}l}  
\text{minimize} & \transpose{c} x \\[\jot]
\text{subject to} & A x + w = b \, , \, x \geq 0 \, , \,  w \geq 0  \\[\jot]
& (x,w) \in \R^n \times \R^m \, ,
\end{array} \qquad \qquad
\shortintertext{and its dual}
\tag*{$\text{DualLP}(A,b,c)$}
& \begin{array}{r@{\quad}l}
\text{maximize}   & - \transpose{b} y \\[\jot]
\text{subject to} & -\transpose{A} y + s = c \, , \, y \geq 0 \, , \, s \geq 0 \\[\jot]
& (y,s) \in \R^m \times \R^n \, ,
\end{array}
\end{align}
where $A \in \R^{m \times n}$, $b \in \R^m$, and $c \in \R^n$.  The primal variables $w$ and the dual variables $s$ are
referred to as \emph{slack variables}.  The reason for picking these particular forms of linear programs is that the
feasible solutions of both are non-negative.  This will allow us to apply our results on linear programs and their
tropicalization discussed in Section~\ref{subsec:tropicalization}.

Let us assume that these two linear programs admit strictly feasible solutions $(x,w)$ and $(y,s)$, \ie~such that
$x,w,y,s>0$; the latter is meant to denote that each coefficient in the respective vectors is strictly positive. In this
situation the system of equations and inequalities
\begin{equation} \label{eq:classical_central_path}
\begin{aligned}
  A x + w & = b \\
-\transpose{A} y +s & =c \\
w_i y_i & = \mu \quad \text{for all} \ i \in [m] \\
x_j s_j & = \mu \quad \text{for all} \ j \in [n] \\
x, w, y, s & > 0 
\end{aligned} 
\end{equation}
is known to have a unique solution $(x^\mu, w^\mu, y^\mu, s^\mu) \in \R^{2(m+n)}$, for any positive real number
$\mu$. The \emph{central path} of the linear programs $\text{LP}(A, b, c)$ and $\text{DualLP}(A, b, c)$ is defined as
the map $\mu \mapsto (x^\mu, w^\mu, y^\mu, s^\mu)$ defined
for positive $\mu$,
and we shall refer to $(x^\mu, w^\mu, y^\mu, s^\mu)$ as
the \emph{point of the central path with parameter} $\mu$.  The equality constraints in
\eqref{eq:classical_central_path} define a real algebraic curve, the \emph{central curve} of the dual pair of linear
programs, which has been studied in \cite{BayerLagarias89a} and \cite{de2010central}.  The central curve is the
Zariski closure of the central path.

The \emph{primal} and \emph{dual central paths} are defined as the projections of the central path onto the $(x, w)$-
and $(y,s)$-coordinates, respectively. Equivalently, given $\mu > 0$, the points $(x^\mu, w^\mu)$ and $(y^\mu, s^\mu)$
on the primal and dual central paths can be defined as the unique optimal solutions of the following pair of
\emph{logarithmic barrier} problems:
\[
\begin{array}{r@{\quad}l}
\text{minimize} & \transpose{c} x - \mu \Bigl(\sum_{j = 1}^n \log(x_j) + \sum_{i = 1}^m \log(w_i)\Bigr) \\
\text{subject to} & A x + w = b \, , \, x > 0 \, , \, w > 0 \, ,
\end{array}
\]
and:
\[
\begin{array}{r@{\quad}l}
\text{maximize} & -\transpose{b} y + \mu \Bigl(\sum_{j = 1}^n \log(s_j) + \sum_{i = 1}^m \log(y_i)\Bigr) \\
\text{subject to} &-\transpose{A} y + s = c \, , \, s > 0 \, , \, y > 0 \, .
\end{array}
\]
The uniqueness of the optimal solutions follows from the fact that the objective functions are 
strictly convex and
concave, 
respectively.  The equivalence to~\eqref{eq:classical_central_path} results from the optimality conditions of
the logarithmic barrier problems.  The main property of the central path is that the sequences
$(x^\mu, w^\mu)$ and $(y^\mu, s^\mu)$ converge to optimal solutions $(x^*, w^*)$ and $(y^*, s^*)$ of the linear programs
$\text{LP}(A,b,c)$ and $\text{DualLP}(A,b,c)$,
when $\mu$ tends to $0$.

Now suppose that the primal feasible set $\{(x,w) \in \R^{n+m} \mid A x + w = b , \ x , w \geq 0\}$ is bounded.  Then,
for $\mu \to +\infty$, the point $(x^\mu, w^\mu)$ tends to the \emph{(primal) analytic center} of that polytope, which
is defined as the unique solution $(x,w) \in (\R_{>0})^{n+m}$ of the following system of equalities:
\begin{align*}
A x + w & = b \\
- \transpose{A} \inverse{w} + \inverse{x} & = 0 \, ,
\end{align*} 
where $\inverse{\cdot}$ denotes the operation which takes the entrywise inverse of a vector, so that
\begin{align*}
\inverse{x}:=\transpose{(x_1^{-1},\dots, x_n^{-1})} \enspace .
\end{align*}
Equivalently, the analytic center is the unique optimal solution of the non-linear optimization problem
\[
\begin{array}{r@{\quad}l}
\text{maximize} & \sum_{j = 1}^n \log(x_j) + \sum_{i = 1}^m \log(w_i) \\
\text{subject to} & A x + w = b \, , \, x > 0 \, , \, w > 0 \, .
\end{array}
\]

\section{The tropicalization of the central path}\label{sec:central_path}

\subsection{Dequantization of a definable family of central paths}
\label{subsect-dequantize}

Our approach to tropicalize the central path starts out with a dual pair of linear programs over the Hardy field $\K$,
which look like the ones studied before:
\begin{align}\tag*{$\textbf{LP}(\bm A,\bm b,\bm c)$}\label{eq:lp}
& \begin{array}{r@{\quad}l}  
\text{minimize} & \transpose{\bm c} \bm x \\[\jot]
\text{subject to} & \bm A \bm x + \bm w = \bm b \, , \, \bm x \geq 0 \, , \, \bm w \geq 0 \\[\jot]
& (\bm x,\bm w) \in \K^n \times \K^m \, ,
\end{array} \\
\shortintertext{where $\bm A \in \K^{m \times n}$, $\bm b \in \K^m$ and $\bm c \in \K^n$, and}
\tag*{$\textbf{DualLP}(\bm A,\bm b,\bm c)$}\label{eq:dual_lp}
& \begin{array}{r@{\quad}l}
\text{maximize} & - \transpose{\bm b} \bm y \\[\jot]
\text{subject to} & -\transpose{\bm A} \bm y + \bm s = \bm c \, , \, \bm s \geq 0 \, , \, \bm y \geq 0 \\[\jot]
& (\bm y, \bm s) \in \K^m \times \K^n \, .
\end{array} 
\end{align}
As in Section~\ref{subsec:central_path_prelim}, we need to assume primal and dual strict feasibility:
\begin{assumption}\label{ass:hardy_central_path_defined}
There exist $(\bm x^\circ, \bm w^\circ)$ and $(\bm y^\circ, \bm s^\circ)$ with positive entries such that $\bm A \bm x^\circ + \bm w^\circ = \bm b$ and $-\transpose{\bm A} \bm y^\circ + \bm s^\circ = \bm c$.
\end{assumption}
Under this assumption, we can show the existence of the central path over the Hardy field: 
\begin{proposition}\label{prop:hardy_central_path}
For all positive $\bm\mu\in \K$, the following system
\begin{equation} \label{eq:hardy_central_path}
\begin{aligned}
  \bm A  \bm x  + \bm w&= \bm b \\
-\transpose{\bm A} \bm y + \bm s &= \bm c \\
\bm w_i \bm y_i &= \bm \mu \quad \text{ for all } i \in [m] \\
\bm x_j \bm s_j &= \bm \mu \quad \text{ for all } j \in [n] \\
\bm x, \bm  w, \bm  y,\bm s& > 0
\end{aligned}
\end{equation}
of equations and inequalities has a unique
solution
$(\bm x^{ \bm \mu}, \bm w ^{\bm \mu},\bm y^{ \bm \mu}, \bm s^{ \bm \mu})$.
\end{proposition}
\begin{proof}
For an ordered field $\F$ and integers $m$ and $n$, consider the following statement:
\begin{quote}
  ``For any $\bm A \in \F^{m \times n}$, $\bm b \in \F^m$ and $\bm c \in \F^n$ satisfying Assumption~\ref{ass:hardy_central_path_defined} and any positive $\bm
  \mu \in \F$, the system~\eqref{eq:hardy_central_path} has a unique solution in $\F^{2(m+n)}$.''
\end{quote}
This is a first-order sentence, $\phi$,
which is true in the structure $\bar \R$, 
that is for $\F = \R$. As $\K$ is
real-closed, Tarski's Principle 
ensures that $\phi$ remains valid with $\F = \K$. 
\end{proof}
The \emph{Hardy central path}
is defined as the map $\bm\mu \mapsto \bm \cpath(\bm \mu):= (\bm x^{ \bm \mu}, \bm w ^{\bm \mu},\bm y^{ \bm \mu}, \bm s^{ \bm \mu})$.
Identifying the germ entries of $\bm A \in \K^{m \times n}$, $\bm b \in \K^m$ and $\bm c \in \K^n$ with any of their
representative functions, the linear programs~\ref{eq:lp} and~\ref{eq:dual_lp} over $\K$ naturally encode a parametric family of linear
programs $\text{LP}(\bm A(t),\bm b(t),\bm c(t))$ and $\text{DualLP}(\bm A(t),\bm b(t),\bm c(t))$ over $\R$.  Assumption~\ref{ass:hardy_central_path_defined} ensures that the central path of these real linear programs is well-defined for all sufficiently large $t$, say $t \geq t_0$ for some $t_0 \in \R$.
Given such a real number $t$, we can define the function
\[\cpath_t : \R \mapsto \R^{2(m+n)}\] 
which maps $\lambda$ to the point on the central path of $\text{LP}(\bm A(t),\bm b(t),\bm c(t))$ with parameter $t^\lambda$. This function constitutes a parameterization of the central path of $\text{LP}(\bm A(t),\bm b(t),\bm c(t))$. Our goal is to investigate the limit when $t \to \infty$ of the following family of functions:
\[
\log_t \cpath_t : \lambda \mapsto \log_t \cpath_t(\lambda) \, ;
\]
here, the function $\log_t$ is applied entrywise. The following result shows that the family of functions $(\log_t \cpath_t)_{t \geq t_0}$ has a point-wise limit.
\begin{lemma}\label{lem:hardy_central_path}
For any $\lambda \in \R$, the map $t \mapsto \cpath_t(\lambda)$ 
from $[t_0,\infty)$ to $\R^{2(m+n)}$ is definable in $\mystruct$. Its germ is given by $\bm \cpath(t^\lambda)$.
In particular, for all $\lambda \in \R$, we have
\[
\lim_{t \to + \infty} \log_t \cpath_t(\lambda) = \val \bm \cpath(t^\lambda)
\enspace .
\]
\end{lemma}

\begin{proof}
For a fixed $\lambda\in \R$, the definability of the map $t\mapsto 
\cpath_t(\lambda)$ follows from the fact that,
for $t\in [t_0,\infty)$, the point
of the central path of $\text{LP}(\bm A(t),\bm b(t),\bm c(t))$
of parameter $t^\lambda$ is defined by a first order formula
in the structure $\mystruct$. 
Let $\bm z =\bm\cpath(t^\lambda)$, so that $\bm z$ is the germ
of a function of the parameter $t$.
Then, for all $t$ large enough, $\bm z(t) \in \R^{2(m+n)}$ %
satisfies the equalities and inequalities defining the point
$\cpath_t(\lambda)$, see~\eqref{eq:classical_central_path}.
Since the latter system has a unique solution, we  conclude that $\cpath_t(\lambda) = \bm z(t)$ for all $t$ large enough.
\end{proof}
Subsequently, we will refer to the point-wise limit of the maps $(\log_t \cpath_t)$ as the \emph{tropical central path}.  By the previous lemma this coincides with the image under the valuation map of the Hardy central path.

\begin{remark}\label{remark:equivalent_assumption}
As over $\R$, the conditions of Assumption~\ref{ass:hardy_central_path_defined} can be equivalently replaced by the fact
that the primal linear program~\ref{eq:lp} is strictly feasible (\ie~there is a feasible point $(\bm x^\circ, \bm w^\circ)$ with positive entries), and the set of its optimal solutions is bounded.
In particular, the latter condition is satisfied when the feasible set of~\ref{eq:lp} is bounded.	
\end{remark}

\subsection{Geometric description of the tropical central path}

The geometric description of the primal tropical central path is more easily obtained by describing the central path
via a logarithmic barrier function. In order to obtain definable barrier functions, we use the structure $\bar
\R_{\exp}$ which expands the ordered real field structure $\bar \R$ by adding the exponential function. The structure
$\bar \R_{\exp}$ is o-minimal \cite{Dries1994}. Note that every power function is definable in $\bar \R_{\exp}$, thus
the definable functions of $\mystruct$ are also definable in $\bar \R_{\exp}$. As a consequence, the Hardy field $
H(\bar \R_{\exp})$ contains $\K =H(\mystruct)$. The logarithm is definable in the structure $\mathfrak H (\bar
\R_{\exp})$ of the Hardy field $H(\bar \R_{\exp})$ using $\exp$, which is a symbol in $\bar\R_{\exp}$. Hence,
if $\bm f \in \K$ is positive, $\log( \bm f) $ belongs to the ordered field $ H(\bar \R_{\exp})$.

Following this, given $\bm \mu \in \K$ such that $\bm \mu > 0$, we define the following primal logarithmic barrier problem
\begin{equation}\label{eq:primal_barrier}
\begin{array}{r@{\quad}l}
\text{minimize} & \transpose{\bm c} \bm x - \bm \mu \Bigl(\sum_{j = 1}^n \log(\bm x_j) + \sum_{i = 1}^m \log(\bm w_i)\Bigr) \\
\text{subject to} & \bm A \bm x + \bm w = \bm b \, , \, \bm x > 0 \, , \, \bm w > 0 \, ,
\end{array}
\end{equation}
and its dual counterpart:
\begin{equation}\label{eq:dual_barrier}
\begin{array}{r@{\quad}l}
\text{maximize} & -\transpose{\bm b} \bm y + \bm \mu \Bigl(\sum_{j = 1}^n \log(\bm s_j) + \sum_{i = 1}^m \log(\bm y_i)\Bigr) \\
\text{subject to} &-\transpose{\bm A} \bm y + \bm s = \bm c \, , \, \bm s > 0 \, , \, \bm y > 0 \, .
\end{array}
\end{equation}

The following lemma relates the central path to the solutions
of the logarithmic barrier problems:
\begin{lemma}\label{lemma:barrier}
Let $\bm \mu \in \K$ such that $\bm \mu > 0$. The two problems~\eqref{eq:primal_barrier} and~\eqref{eq:dual_barrier} both have a unique solution, respectively given by the points $(\bm x^{\bm \mu}, \bm w^{\bm \mu})$ and $(\bm y^{\bm \mu}, \bm s^{\bm \mu})$.
\end{lemma}

\begin{proof}

  As noted in Section~\ref{subsec:central_path_prelim}, given $m$ and $n$, the following statement
  \begin{quote}
    ``For any $\bm A \in \F^{m \times n}$, $\bm b \in \F^m$ and $\bm c \in \F^n$ which satisfy
    Assumption~\ref{ass:hardy_central_path_defined} and any positive $\bm \mu \in \F$, the optimization
    problems~\eqref{eq:primal_barrier} and~\eqref{eq:dual_barrier} both have a unique solution, respectively given by the points $(\bm x, \bm w)$ and $(\bm y, \bm s)$ where $(\bm x, \bm w, \bm y, \bm s)$ is the unique solution
    of~\eqref{eq:hardy_central_path}.''
  \end{quote}
  is a valid sentence in the structure $\bar\R_{\exp}$ for $\F = \R$.  Since the structure $\bar\R_{\exp}$ is o-minimal,
  by Proposition~\ref{prop:model_complete}, that sentence is also true in $\mathfrak H (\bar\R_{\exp})$, \ie~for $\F =
  H( \bar \R_{\exp})$. It follows that the sentence is valid when $\bm A$, $\bm b$, $\bm c$ and $\bm \mu$ have entries
  in $\K \subset H(\bar \R_{\exp})$.
\end{proof}

A set $C \subset \trop^n$ is said to be \emph{tropically convex} if $\alpha \tdot u \tplus \beta \tdot v \in C$ as soon
as $u,v \in C$ and the tropical sum of $\alpha, \beta \in \trop$ is equal to the tropical unit $0$. If $C$ is a
non-empty tropically convex set, the supremum $\sup(u,v) = u \tplus v$ with respect to the partial order of $\trop^n$
also belongs to $C$. If in addition $C$ is compact, then the supremum of an arbitrary subset of $C$ is well-defined and
belongs to $C$. Consequently, there is a unique element in $C$ which is the coordinate-wise maximum of all elements in
$C$. We call it the \emph{(tropical) barycenter} of $C$, as it is the mean of $C$ with respect to the uniform idempotent
measure.  In particular, any non-empty and bounded tropical polyhedron has a tropical barycenter, since it is compact and
tropically convex.

We denote by $\puiseuxP := \{(\bm x, \bm w) \in \K^{n+m} \mid \bm A \bm x + \bm w = \bm b, \ \bm x, \bm w \geq 0 \}$ and $\puiseuxQ := \{(\bm y, \bm s) \in \K^{m+n} \mid - \transpose{\bm A} \bm y + \bm s = \bm c, \ \bm y, \bm s \geq 0\}$ the sets of feasible solutions of~\ref{eq:lp} and~\ref{eq:dual_lp} respectively. Since none of these sets are empty (thanks to Assumption~\ref{ass:hardy_central_path_defined}), the two linear programs~\ref{eq:lp} and~\ref{eq:dual_lp} have the same optimal value $\bm \nu \in \K$, and they admit optimal solutions $(\bm x^*, \bm w^*)$ and $(\bm y^*, \bm s^*)$, respectively. This is a consequence of strong duality, which is still valid over the ordered field $\K$.

Let us define $\tropP := \val(\puiseuxP)$, $\tropQ := \val(\puiseuxQ)$, $(x^*, w^*) := \val(\bm x^*, \bm w^*)$ and $(y^*, s^*) := \val (\bm y^*, \bm s^*)$. We introduce the map $\troppath : \R \to \R^{2(m+n)}$ defined by $\troppath(\lambda) := (x^\lambda, w^\lambda, y^\lambda, s^\lambda)$, where $(x^\lambda, w^\lambda)$ is the barycenter of the tropical polyhedron
\[
\tropP^\lambda := \tropP \cap \Bigl\{(x,w) \in \trop^{n+m} \mid \bigl(\transpose{(s^*)} \tdot x\bigr) \tplus \bigl(\transpose{(y^*)} \tdot w\bigr) \leq \lambda \Bigr\}\, ,
\]
and similarly, $(y^\lambda, s^\lambda)$ is the barycenter of the tropical polyhedron:
\[
\tropQ^\lambda := \tropQ \cap \Bigl\{(y,s) \in \trop^{m+n} \mid \bigl(\transpose{(w^*)} \tdot y \bigr) \tplus \bigl(\transpose{(x^*)} \tdot s \bigr) \leq \lambda \Bigr\}\, .
\]
We point out that the quantities $ \bigl(\transpose{(s^*)} \tdot x\bigr) \tplus \bigl(\transpose{(y^*)} \tdot w\bigr) $
and $\bigl(\transpose{(w^*)} \tdot y \bigr) \tplus \bigl(\transpose{(x^*)} \tdot s \bigr)$ can be interpreted as the
tropical analogues of the optimality gaps in the primal and the dual setting, respectively.  More precisely, if $(\bm x,
\bm w) \in \puiseuxP$, the gap between the value of the objective function at $(\bm x, \bm w)$ and the optimal value is
given by
\begin{equation}\label{eq:primal_gap}
\transpose{\bm c} \bm x - \bm \nu = \transpose{(\bm s^*)} \bm x + \transpose{(\bm y^*)} \bm w \, ,
\end{equation}
as a consequence of the fact that $\bm \nu = -\transpose{\bm b} \bm y^*$. Since all the terms in the right-hand side of~\eqref{eq:primal_gap} are non-negative, we deduce that 
\begin{equation}\label{eq:tropical_gap}
\val(	\transpose{\bm c} \bm x - \bm \nu) = \bigl(\transpose{(s^*)} \tdot x\bigr) \tplus \bigl(\transpose{(y^*)} \tdot w\bigr) \ ,
\end{equation}
where $(x,w) = \val(\bm x, \bm w) \in \tropP$. We now prove the following characterization:
\begin{theorem}\label{th:geometric_characterization}
The tropical central path coincides with the map $\troppath$. 
\end{theorem}

\begin{proof}
Let us fix $\lambda \in \R$, and let $\bm \mu := t^\lambda$. We restrict our attention to the proof of $\val(\bm x^{\bm \mu}, \bm w^{\bm \mu}) = (x^\lambda, w^\lambda)$, as the proof of $\val(\bm y^{\bm \mu}, \bm s^{\bm \mu}) = (y^\lambda, s^\lambda)$ is similar.

First, we point out that the primal logarithmic barrier problem~\eqref{eq:primal_barrier} is equivalent to minimizing
the following function
\[
f : (\bm x,\bm w) \mapsto \frac{\transpose{\bm c} \bm x - \bm \nu}{\bm \mu} - \Bigl(\sum_{j = 1}^n \log(\bm x_j) + \sum_{i = 1}^m \log(\bm w_i) \Bigr)
\]
over the elements $(\bm x, \bm w) \in \puiseuxP$ which additionally satisfy $\bm x > 0$, $\bm w > 0$. Recall that $f(\bm x, \bm w)$ is the germ of a  function definable in the structure $\bar \R_{\exp}$, so that it makes sense to consider the real number $[f(\bm x, \bm w)](t)$ for $t$ sufficiently large.

Let us consider $(\bm x, \bm w) \in \puiseuxP$ such that $\bm x, \bm w > 0$.  We distinguish two cases. First, assume
$\val(\transpose{\bm c} \bm x - \bm \nu) \leq \lambda$. We call a point $(\bm x,\bm w)$ with this property a point of
the \emph{first kind}.  Since $\transpose{\bm c} \bm x \geq \bm \nu$, the term $[(\transpose{\bm c} \bm x - \bm \nu)/\bm
\mu](t)$ in $[f(\bm x, \bm w)](t)$ is asymptotically of the form $p t^\alpha + o(t^\alpha)$, for some $\alpha, p \in \R$
with $\alpha \leq 0$ and $p \geq 0$. Moreover, the terms $[\log \bm x_j](t)$ and $[\log \bm w_i](t)$ are of the form
$\val(\bm x_j) \log t + O(1)$ and $\val(\bm w_i) \log t + O(1)$. As a consequence, we can write
\begin{equation}\label{eq:asymptotic}
\bigl[f(\bm x, \bm w)\bigr](t) = - \Bigl(\sum_{j = 1}^n \val(\bm x_j) + \sum_{i = 1}^m \val(\bm w_i)\Bigr) \log t + O(1) \, .
\end{equation}
Observe that a point of the first kind always exists.
Indeed, by Assumption~\ref{ass:hardy_central_path_defined}, there is a strictly feasible
point $(\bm x^\circ, \bm w^\circ)$. 
Let $\bm \gamma \in \K$ be such that $0 \leq \bm \gamma \leq 1$,
with a sufficiently small valuation. Then the point
$(\bm x, \bm w) = (1 - \bm \gamma) (\bm x^*, \bm w^*) + \bm \gamma (\bm x^\circ, \bm w^\circ)$ is of the first kind.

Consider now a point $(\bm x,\bm w)$ of the \emph{second kind}, meaning that $\val(\transpose{\bm c} \bm x - \bm \nu) >
\lambda$.  The term which asymptotically dominates in $[f(\bm x, \bm w)](t)$ is given by $[(\transpose{\bm c} \bm x -
\bm \nu)/\bm \mu](t)$, since it is of the form $p t^{\alpha} + o(t^{\alpha})$, with $\alpha > 0$ and $p > 0$.  By
comparing the asymptotics of $[f(\bm x, \bm w)](t)$ for points of the first and of the second kinds, we deduce that the
minimization problem~\eqref{eq:primal_barrier} over the entire Hardy polyhedron $\puiseuxP$ attains its unique optimum
$(\bm x^{\bm \mu}, \bm w^{\bm \mu})$ at a point of the first kind.

Let $\puiseuxS$ be the set of points of the first kind.  By~\eqref{eq:asymptotic}, we infer that $\val(\bm x^{\bm \mu},
\bm w^{\bm \mu})$ necessarily maximizes the function
\[ \phi : (x,w) \mapsto \sum_{j = 1}^n x_j + \sum_{i = 1}^m w_i
\]
as $(x,w)$ ranges over the set $\val(\puiseuxS)$. Using~\eqref{eq:tropical_gap}, we have $\val(\puiseuxS) =
\tropP^\lambda \cap \R^{m+n}$, hence $\val(\bm x^{\bm \mu}, \bm w^{\bm \mu})$ actually maximizes $\phi$ over the
tropical polyhedron $\tropP^\lambda$. It follows that the set $\tropP^\lambda$ is necessarily bounded, since otherwise,
the function $\phi$ would be unbounded.  In particular, the tropical barycenter $(x^\lambda, w^\lambda)$ of
$\tropP^\lambda$ is well defined. The latter point is the unique maximizer of the function $\phi$ over
$\tropP^\lambda$. We conclude that $\val(\bm x^{\bm \mu}, \bm w^{\bm \mu}) = (x^\lambda, w^\lambda)$.
\end{proof}

We define the \emph{primal tropical central path} 
as the projection of the tropical central path onto the $(m+n)$ first
coordinates, or equivalently, as the function which maps $\lambda \in \R$ to $(x^\lambda, w^\lambda)$.
The \emph{dual tropical central path} is defined similarly.
We point out that the dual tropical central path is completely determined by the primal one (and vice versa) by the relations:
\begin{equation}
x^\lambda_j \tdot s^\lambda_j = w^\lambda_i \tdot y^\lambda_i = \lambda \label{eq:primal_dual}
\end{equation}
for all $\lambda \in \R$, $i \in [m]$ and $j \in [n]$. This is a consequence of the fact that $\val (\bm \cpath(\lambda)) = \troppath(\lambda)$ by 
Theorem~\ref{th:geometric_characterization}, and that the point $\bm \cpath(\lambda) = (\bm x, \bm w, \bm y, \bm s)$ satisfies $\bm x_j \bm s_j = \bm w_i \bm y_i = t^\lambda$ for all $i,j$. Applying the valuation map to these relations yields~\eqref{eq:primal_dual}.

The following corollary exhibits a special case in which the primal tropical central path has a remarkable formulation in terms of sublevel sets of the objection function.
\begin{corollary}\label{cor:geometric_characterization}
  Let $\bm c \geq 0$, and suppose that the optimal value of the dual pair of linear programs~\ref{eq:lp}
  and~\ref{eq:dual_lp} equals $0$.

  Then, for all $\lambda \in \R$, the primal component $(x^\lambda, w^\lambda)$ of the tropical central path is given by
  the barycenter of the tropical polyhedron $\tropP \cap \{(x, w) \in \trop^{n+m} \mid \transpose{c} \tdot x \leq
  \lambda \}$, where $c := \val(\bm c)$.
\end{corollary}

\begin{proof}
  Using~\eqref{eq:tropical_gap}, given $(x, w) \in \tropP$ and $(\bm x, \bm w) \in \puiseuxP$ such that $(x,w) =
  \val(\bm x, \bm w)$, we know that $\val(\transpose{\bm c} \bm x) = \bigl(\transpose{(s^*)} \tdot x\bigr) \tplus
  \bigl(\transpose{(y^*)} \tdot w\bigr)$. As $\bm c$ and $\bm x$ both have non-negative components, we have
  $\val(\transpose{\bm c} \bm x) = \transpose{c} x$. The claim now follows from
  Theorem~\ref{th:geometric_characterization}.
\end{proof}
An analogous statement can be derived when $\bm b \geq 0$ for the dual component $(y^\lambda, s^\lambda)$ of the
tropical central path.

Similar to the classical case, if the Hardy polyhedron $\puiseuxP$ is bounded, we can define the \emph{analytic center}
of $\puiseuxP$ as the unique solution $(\bm x, \bm w) \in \K^{n+m}$ of the system
\begin{align*}
\bm A \bm x + \bm w & = \bm b \\
- \transpose{\bm A} \inverse{\bm w} + \inverse{\bm x} & = 0 \\
\bm x, \bm w > 0 
\end{align*}
of equations and inequalities, recalling that $\inverse{\cdot}$ denotes the entrywise inverse.  This is due to the fact that $\K$ is real-closed. Using the arguments of the proof of
Lemma~\ref{lemma:barrier}, we can prove that the analytic center is the unique maximizer of the function $(\bm x, \bm w)
\mapsto \sum_{j = 1}^m \log \bm x_j + \sum_{i = 1}^m \log \bm w_i$ over the set $\puiseuxP$. The germ of the analytic
center of the polytopes $\puiseuxP(t)$ (with $t$ large enough) is precisely the analytic center of $\puiseuxP$. By
taking a sufficiently large $\lambda$ in the characterization of Theorem~\ref{th:geometric_characterization}, we deduce the following corollary:
\begin{corollary}\label{coro:analytic_center}
Suppose that the polyhedron $\puiseuxP$ is bounded. Then, the image under the valuation map of the analytic center of $\puiseuxP$ coincides with the tropical barycenter of $\tropP = \val(\puiseuxP)$.
\end{corollary}
We point out that, even if the analytic center depends on the inequality representation of the set $\puiseuxP$, its
tropical analogue is, surprisingly, completely determined by the set $\puiseuxP$.

\begin{ex}\label{ex:ex1}
Consider the Hardy polyhedron $\puiseuxR \subset \K^2$ defined by:
\begin{equation}\label{eq:ex}
\begin{aligned}
  \bm x_1 + \bm x_2 & \leq 2 \\
  t \bm x_1 & \leq 1 + t^2 \bm x_2 \\
  t \bm x_2  & \leq 1 + t^3 \bm x_1 \\
  \bm x_1 & \leq t^2 \bm x_2  \\
\bm x_1, \bm x_2 & \geq 0 \, .
\end{aligned}
\end{equation}
The tropical polyhedron $\tropR = \val(\puiseuxR)$ is described by the inequalities:
\begin{equation}\label{eq:ex_trop}
\begin{aligned}
  \max (x_1 , x_2) & \leq 0 \\
  1 + x_1 & \leq \max(0, 2+ x_2)\\
  1 + x_2 & \leq \max(0, 3 + x_1 ) \\
  x_1  & \leq 2 + x_2 \, .
\end{aligned}
\end{equation}
Notice that our inequality descriptions omit the slack variables.  The primal tropical central path associated with
$\tropR$ and the two objective functions $\bm x \mapsto \bm x_2$ and $\bm x \mapsto t \bm x_1 + \bm x_2$ is depicted in
Figure~\ref{fig:path1}.  Both objective functions are non-negative with optimal value zero.  So the primal components of
the respective tropical central paths are described by Corollary~\ref{cor:geometric_characterization}; the tropical
barycenter of $\tropR$ is the origin.

The tropical constraints \eqref{eq:ex_trop} give rise to an arrangement of tropical halfspaces which induces a cell
decomposition of $\R^2$; see \cite[\S3.6]{JoswigLoho:1503.04707}. Figure~\ref{fig:path2} depicts the tropical central
paths in each cell of this arrangement for the objective function $\bm x \mapsto t \bm x_1 + \bm x_2$.  Observe that the
central paths trace the arrangement of tropical hyperplanes associated with the tropical halfspaces
in~\eqref{eq:ex_trop}, as well as the line $\{ (-1+ \gamma, \gamma) \mid \gamma \in \R \}$ associated with the
objective function.
\end{ex}

\newcommand{\axisandedges}{
   \draw[axis, ->] (-4, -4) -- (1.3, -4);
   \draw[axis, ->] (-4, -4) -- (-4, 1.3);
   
   \foreach \x in {-4, -3, ..., 0} { 
     \node [anchor=north] at (\x,-4.1) {${\scriptstyle\x}$}; 
     \draw[axis] (\x, -4.1) -- (\x, -3.9);
   }
   \foreach \y in {-4, -3, ..., 0} { 
     \node [anchor=east] at  (-4.1,\y) {${\scriptstyle\y}$};
     \draw[axis] (-4.1, \y) -- (-3.9, \y);
   }

   \node[below] at (1.3, -4) {$x_1$};
   \node[left] at (-4, 1.3) {$x_2$};
    \draw[edge] (-4,0) -- (0,0) -- (0,-4);
    \draw[edge] (-4, -1) -- (-3, -1) -- (-1,1);
    \draw[edge] (-2,-4) -- (1,-1);
    \draw[edge] (-1,-4) -- (-1,-2) -- (1,0);
}

\begin{figure}\label{fig:ex1}
  \begin{tikzpicture}
[
    axis/.style={lightgray},
    edge/.style={black},
    poly/.style={fill=lightgray, fill opacity = 0.7},
    path/.style={blue, very thick}
]

  \begin{scope}[shift = {(-5,0)}]
   \fill[poly] (0,0) -- (0,-1) -- (-1,-2) -- (-1,-3) -- (-2,-4) -- (-4,-4)
                -- (-4,-1) -- (-3, -1) -- (-2, 0) -- cycle;
\axisandedges

\draw[path] (0,0) -- (0,-1) -- (-1,-2) -- (-1,-3) -- (-2,-4);

    \end{scope}

  \begin{scope}[shift = {(2,0)}]
   \fill[poly] (0,0) -- (0,-1) -- (-1,-2) -- (-1,-3) -- (-2,-4) -- (-4,-4)
                -- (-4,-1) -- (-3, -1) -- (-2, 0) -- cycle;
\axisandedges
\draw[path] (0,0) -- (-1,0) -- (-4,-3);
    \end{scope}

\end{tikzpicture}
  \centering
  \caption{Tropical central paths associated with the Hardy polyhedron given in~\eqref{eq:ex} and the objective function $\bm x \mapsto \bm x_2$ (left) and  $\bm x \mapsto t \bm x_1 + \bm x_2$ (right).}\label{fig:path1}
\end{figure}
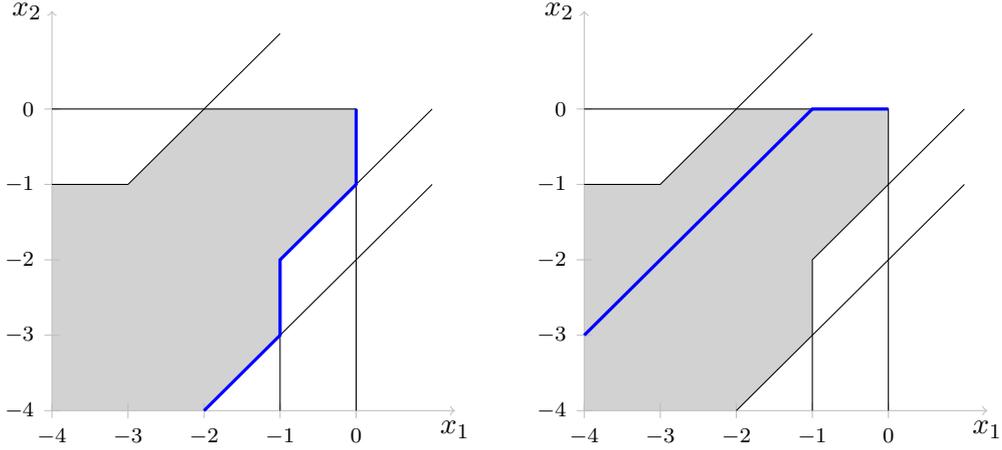

\begin{figure}\label{fig:ex2}
  \begin{tikzpicture}
[
    axis/.style={lightgray},
    edge/.style={black},
    poly/.style={fill=lightgray, fill opacity = 0.7},
    path/.style={blue, very thick, rounded corners},
]

\newcommand{\e}{0.05}

  \begin{scope}[shift = {(-5,0)}]
    
\axisandedges

\draw[path] (-\e,-\e) -- (-1,-\e) -- (-4,-3-\e);

\draw[path] (-2+2*\e, \e) -- (-1, \e) -- (0, 1+\e);

\draw[path] (-4, \e) -- (-2-\e, \e) -- (-1 - \e, 1+\e);

\draw[path] (-4, -1+\e) -- (-3 -\e, -1+ \e) -- ( -3 - \e +1 -2*\e  , - \e );

\draw[path] (-1-\e, -3 - 3*\e) -- (-2+2*\e, -4);

\draw[path] (-1+\e, -4) -- (-1+\e,-3 - \e) -- (-\e, -2 - 3*\e);

\draw[path] (-1+\e, -3+ 2*\e) -- (-1+\e,-2 - \e) -- (-\e, -1 - 3*\e);

\draw[path]  (1, -2*\e) -- (\e, -1-\e) -- (\e, -2+2*\e);

\draw[path]  (1, -1-2*\e) -- (\e, -2-\e) -- (\e, -4);

    \end{scope}

\end{tikzpicture}

  \centering
  \caption{Tropical central paths in the full-dimensional cells of the arrangement of tropical halfspaces associated
    arising from the Hardy linear inequalities \eqref{eq:ex} and the objective function $\bm x \mapsto t \bm x_1 + \bm
    x_2$. For better visibility the parts of the paths that lie on the boundaries are slightly shifted inside their
    respective cell.}
  \label{fig:path2}
\end{figure}
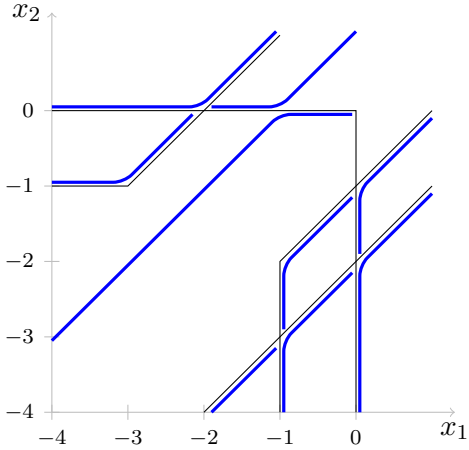

\subsection{Uniform convergence and metric estimates}

In this section, we show that the convergence of the functions $(\log_t \cpath_t)_t$ to the map $\troppath$ is uniform, and we establish an upper bound on the sup metric 
\[ d_\infty(\troppath, \log_t \cpath_t) := \sup_{\lambda \in \R} \,\bigl\lvert \troppath(\lambda) - \log_t \cpath_t(\lambda) \bigr\rvert\]
between $\troppath$ and $\log_t \cpath_t$. It is worth noting that the proof of the uniform convergence given below is independent of the previous results in Section~\ref{sec:central_path}, such as the existence of the Hardy central path, the point-wise convergence of $(\log_t \cpath_t)_t$, and the characterization of the limit given in Theorem~\ref{th:geometric_characterization}. Moreover, the proof only relies on the fact that $\K$ is an ordered field.
Beyond this property, we do not exploit further results 
in model theory. This is sufficient to ensure that the basic results involved in linear programming (Minkowski--Weyl theorem, Strong Duality, etc) are valid. 

We introduce the following non-symmetric metric, 
\[
\funk(x,y) := \inf\bigl\{ \rho \geq 0 \mid \rho \tdot x \geq y \bigr\} \, ,
\]
where $x, y \in \trop^{m+n}$. This is a \emph{hemi-metric}
in the sense of~\cite{DezaDeza}.
The function $\funk$
 is a tropical analogue of the \emph{Funk metric} which appears in Hilbert's
geometry~\cite{PT14}.  Note that $\funk(x,y) < + \infty$ if and only if 
$x_k = -\infty$
implies $y_k = -\infty$,
for all $k \in [m+n]$.
More precisely, $\funk(x,y) = \max(0,\max_k (y_k - x_k))$, with the convention $-\infty +
\infty = +\infty$.  We shall also use the following symmetrization of $\funk$:
\[
\hilbert(x,y) := \funk(x,y) + \funk(y,x) \, . 
\]
This is an affine version of \emph{Hilbert's projective metric}, which was shown to be in some sense the canonical
metric in tropical convexity~\cite{cgq02}. The relevance of Hilbert's geometry to the study of the central path
was already observed by Bayer and Lagarias~\cite{BayerLagarias89a}. 
Observe that $\hilbert(x,y) < + \infty$ if and only if the
two sets $\{ k \in [m+n] \mid x_k = -\infty \}$ and $\{ k \in [m+n] \mid y_k = -\infty \}$ are identical. By abuse of
notation, given two sets $X, Y \subset \trop^{m+n}$, we denote by $\hilbert(X,Y)$ the directed Hausdorff distance from $X$ to
$Y$ induced by $\hilbert$, \ie~$\hilbert(X,Y) := \sup_{x \in X} \inf_{y \in Y} \hilbert(x,y)$.

Let us define the polyhedra $\puiseuxP(t) := \{ (x,w) \in \R^{n+m} \mid \bm A(t) x + w = \bm b(t), \ x, w \geq 0\}$ and
$\puiseuxQ(t) := \{(y, s) \in \R^{m+n} \mid - \transpose{\bm A(t)} y + s = \bm c(t), \ y, s \geq 0 \}$ over $\R$. We
introduce the quantity
\begin{multline*}
\delta(t) := \log_t (m+n) + \max\Bigl(\hilbert(\log_t \puiseuxP(t), \tropP) + \funk\bigl(\log_t (\bm y^*(t), \bm s^*(t)), (y^*, s^*)\bigr) ,  \\
\hilbert(\log_t \puiseuxQ(t), \tropQ) + \funk\bigl(\log_t (\bm x^*(t), \bm w^*(t)), (x^*, w^*)\bigr)\Bigr) \, ,
\end{multline*}
which bounds the distance between the tropical polyhedron $\tropP$ and the image of the classical polyhedron
$\puiseuxP(t)$ under the map $\log_t$ and simultaneously the distance between $\tropQ$ and $\log_t\puiseuxQ(t)$.  As we
will show in Lemma~\ref{lemma:delta} below, the expression $\delta(t)$ tends to $0$ when $t$ goes to $+\infty$.

\begin{theorem}\label{th:uniform}
  The family of functions $(\log_t \cpath_t)_t$ converges uniformly to the map $\troppath$ for $t \to +\infty$. More
  precisely, for all $t$ sufficiently large, we have
  \begin{equation}\label{eq:uniform}
    d_\infty(\troppath, \log_t \cpath_t) \leq \delta(t) \, . 
  \end{equation}
\end{theorem}

We establish a few properties on the map $\troppath$ which will be useful in the proof of Theorem~\ref{th:uniform}. Since we do not rely on the results of the previous section, we give an independent proof that the points $(x^\lambda, w^\lambda)$ and $(y^\lambda, s^\lambda)$ are well-defined, and that they satisfy a duality property. The latter is weaker than the one given in~\eqref{eq:primal_dual}, for now. We also show a regularity property of~$\troppath$. 
\begin{lemma}\label{lemma:troppath}
Let $\lambda \in \R$. The following properties hold:
\begin{enumerate}[(i)]
\item\label{item:troppath0} the tropical polyhedra $\tropP^\lambda$ and $\tropQ^\lambda$ are compact;
\item\label{item:troppath1} $x^\lambda_j \tdot s^\lambda_j \leq \lambda$ and $w^\lambda_i \tdot y^\lambda_i \leq \lambda$ for all $i \in [m]$, $j \in [n]$;
\item\label{item:troppath2} for all $\lambda' \geq \lambda$, we have $\troppath(\lambda)_k \leq \troppath(\lambda')_k \leq \troppath(\lambda)_k + (\lambda' - \lambda)$ for every $k \in [2(m+n)]$. 
\end{enumerate}
\end{lemma}

\begin{proof}
\begin{asparaenum}[(i)]
\item Given $(\bm x, \bm w) \in \puiseuxP$ and $(\bm y, \bm s) \in \puiseuxQ$, we have
\begin{align*}
  \transpose{\bm s} \bm x + \transpose{\bm y} \bm w & = \transpose{\bm c} \bm x + \transpose{\bm b} \bm y = (\transpose{\bm c} \bm x - \bm \nu) + (\transpose{\bm b} \bm y + \bm \nu) \\
  & = \bigl(\transpose{(\bm s^*)} \bm x + \transpose{(\bm y^*)} \bm w\bigr) + \bigl(\transpose{(\bm w^*)} \bm y +
  \transpose{(\bm x^*)} \bm s\bigr) \, .
\end{align*}
Since all the terms in the previous identity are non-negative, applying the valuation map shows that for all $(x,w) \in \tropP^\lambda$ and $(y,s) \in \tropQ^\lambda$,
\begin{equation}\label{eq:bounded}
(\transpose{s} \tdot x) \tplus (\transpose{y} \tdot w) =
\bigl(\transpose{(s^*)} \tdot x \bigr) \tplus \bigl(\transpose{(y^*)} \tdot w \bigr) \tplus
\bigl(\transpose{(w^*)} \tdot y \bigr) \tplus \bigl(\transpose{(x^*)} \tdot s \bigr) \leq \lambda \, .
\end{equation}
Equivalently, $x_j \tdot s_j \leq \lambda$ and $w_i \tdot y_i \leq \lambda$ for all $i \in [m]$ and $j \in [n]$. Provided that we can find a point $(y,s) \in \tropQ^\lambda$ such that $y_i, s_j > -\infty$ for all $i,j$, we deduce that the tropical polyhedron $\tropP^\lambda$ is bounded. We now exhibit such a point. Since $\transpose{(\bm w^*)} \bm y^* + \transpose{(\bm x^*)} \bm s^* = \bm \nu - \bm \nu = 0$, we know that $\bigl(\transpose{(w^*)} \tdot y^* \bigr) \tplus \bigl(\transpose{(x^*)} \tdot s^* \bigr) = -\infty$. Therefore, $(y^*, s^*) \in \tropQ^\lambda$. Besides, if we define $(y^\circ, s^\circ) := \val(\bm y^\circ, \bm s^\circ)$, the point $(y^*, s^*) \tplus \bigl(\alpha \tdot (y^\circ, s^\circ)\bigr)$ belongs to $\tropQ$ for all $\alpha \in \R_{\leq 0}$ because $\tropQ$ is tropically convex. The latter point belongs to $\tropQ^\lambda$ as soon as $\alpha$ is small enough, and none of its entries is equal to $-\infty$. 

We can prove that $\tropQ^\lambda$ is bounded by a symmetric argument.

\item The property comes from~\eqref{eq:bounded} applied to the points $(x^\lambda, w^\lambda)$ and $(y^\lambda, s^\lambda)$.

\item Since $\lambda\leq\lambda'$ we have $\tropP^\lambda \subset \tropP^{\lambda'}$, and this implies $\troppath(\lambda) \leq \troppath(\lambda')$ componentwise.  Now consider $(x,w) := (x^\lambda, w^\lambda) \tplus \bigl((\lambda - \lambda') \tdot (x^{\lambda'},w^{\lambda'})\bigr)$. That point lies in $\tropP$ since $(x^\lambda, w^\lambda), (x^{\lambda'},w^{\lambda'}) \in \tropP$ and $\tropP$ is tropically convex. Further, $\bigl(\transpose{(s^*)} \tdot x\bigr) \tplus \bigl(\transpose{(y^*)} \tdot w\bigr) \leq \lambda$. 
We deduce that $(x,w) \leq (x^\lambda,w^\lambda)$, which ensures that $x^{\lambda'}_j \leq x^{\lambda}_j + (\lambda' - \lambda)$ and $w^{\lambda'}_i \leq w^{\lambda}_i + (\lambda' - \lambda)$ for all $i \in [m]$ and $j \in [n]$. By using similar arguments for $(y^\lambda, s^\lambda)$ and $(y^{\lambda'}, s^{\lambda'})$, we obtain that $(\troppath(\lambda'))_k \leq (\troppath(\lambda))_k + (\lambda' - \lambda)$ holds for all $k \in [2(m+n)]$. \qedhere
\end{asparaenum}
\end{proof}

\begin{proof}[\proofname~(Theorem~\ref{th:uniform})]
  For $t$ sufficiently large, say $t \geq t_0$ without loss of generality, the points $(\bm x^*(t), \bm w^*(t))$ and
  $(\bm y^*(t), \bm s^*(t))$ are optimal solutions of the linear programs $\text{LP}(\bm A(t), \bm b(t), \bm c(t))$ and
  $\text{DualLP}(\bm A(t), \bm b(t), \bm c(t))$ respectively, and the optimal value of the two linear programs is equal
  to $\bm \nu(t)$. Further, there exists a real number $t_1$ such that for all $t\geq t_1$, we have $\log_t (\bm y^*(t),
  \bm s^*(t))_k = -\infty$ if and only if $(y^*, s^*)_k = -\infty$. Indeed, both are equivalent to $(\bm y^*, \bm s^*)_k
  = 0$.  It follows that $\funk\bigl(\log_t (\bm y^*(t), \bm s^*(t)), (y^*, s^*)\bigr) < +\infty$ for all $t \geq
  t_1$.

Now, let us fix $t \geq \max(t_0,t_1)$ and $\lambda \in \R$. We want to show that $\bigl|\troppath(\lambda)_k - \log_t \cpath_t(\lambda)_k \bigr| \leq \delta(t)$ for all $k \in [2(m+n)]$. We claim that it suffices to prove that
  \begin{equation}\label{eq:upper_bound}
    \log_t \cpath_t(\lambda)_k \leq \troppath(\lambda)_k + \delta(t) \, .
  \end{equation}
Indeed, let us set $(x,w,y,s) := \cpath_t(\lambda)$. For all $k \in [m+n]$, we have:
  \[
  \log_t (x, w)_k = \lambda - \log_t (s, y)_k \geq \lambda - (s^\lambda, y^\lambda)_k - \delta(t) \geq (x^\lambda, w^\lambda)_k - \delta(t) \, ,
  \]
  where the second inequality is given by Lemma~\ref{lemma:troppath}~\eqref{item:troppath1}. Similarly, we can prove that $\log_t (y,s)_k \geq (y^\lambda, s^\lambda)_k - \delta(t)$. This proves the claim.

  Finally, let us show that~\eqref{eq:upper_bound} holds. Recall that $(x,w,y,s) = \cpath_t(\lambda)$. As $\transpose{(\bm b(t))} \bm y^*(t) \leq \transpose{(\bm b(t))} y$, we have:
  \begin{align*}
    \transpose{(\bm s^*(t))} x + \transpose{(\bm y^*(t))} w  & = \transpose{(\bm c(t))} x + \transpose{(\bm b(t))} \bm y^*(t)\\
    & \leq \transpose{(\bm c(t))} x + \transpose{(\bm b(t))} y = \transpose{s} x + \transpose{y} w = (m+n) t^\lambda \, .
  \end{align*}
  Hence, $(\log_t \bm s^*_j(t)) \tdot (\log_t x_j) \leq \lambda + \log_t (m+n)$ and $(\log_t \bm y^*_i(t)) \tdot (\log_t
  w_i) \leq \lambda + \log_t (m+n)$ for all $i \in [m]$ and $j \in [n]$. Since $(y^*, s^*)_k \leq \log_t (\bm y^*(t),
  \bm s^*(t))_k + \funk\bigl(\log_t (\bm y^*(t), \bm s^*(t)), (y^*, s^*)\bigr)$ for each $k \in [m+n]$, we obtain:
  \begin{equation}
    \bigl(\transpose{(s^*)} \tdot \log_t x\bigr) \tplus \bigl(\transpose{(y^*)} \tdot \log_t w\bigr) \leq \lambda + \log_t (m+n) + \funk\bigl(\log_t (\bm y^*(t), \bm s^*(t)), (y^*, s^*)\bigr)\, . \label{eq1}
  \end{equation}
  Due to Assumption~\ref{ass:hardy_central_path_defined}, we know that there exists a point $(x^\circ, w^\circ) \in
  \tropP$ such that $x^\circ_j, w^\circ_i > -\infty$. As $\log_t(x,w)$ has no $-\infty$ entries, we deduce that
  $\hilbert(\log_t (x, w), \tropP) \leq \hilbert(\log_t (x,w), (x^\circ, w^\circ)) < +\infty$. Let $(\bar x, \bar w) \in \tropP$
  such that $\hilbert(\log_t (x, w), (\bar x, \bar w)) < +\infty$, so that $\bar x_j, \bar w_i > -\infty$. We derive
  from~\eqref{eq1} that $(\bar x, \bar w) \in \tropP^{\lambda'}$, where
  \[
  \lambda' := \lambda + \log_t(m+n) + \funk\bigl(\log_t (\bm y^*(t), \bm s^*(t)), (y^*, s^*)\bigr) + \funk\bigl(\log_t (x, w), (\bar x, \bar w)\bigr) \, .
  \]
  Hence, for all $k \in [m+n]$,
  \begin{align*}
    \log_t (x,w)_k & \leq (\bar x, \bar w)_k + \funk\bigl((\bar x, \bar w), \log_t (x, w)\bigr) \\
    & \leq (x^{\lambda'}, w^{\lambda'})_k + \funk\bigl((\bar x, \bar w), \log_t (x, w)\bigr)  \\
    & \leq (x^{\lambda}, w^{\lambda})_k + \log_t(m+n) + \funk\bigl(\log_t (\bm y^*(t), \bm s^*(t)), (y^*, s^*)\bigr) +\hilbert\bigl(\log_t (x, w), (\bar x, \bar w)\bigr) 
  \end{align*}
  where the last inequality comes from Lemma~\ref{lemma:troppath}~\eqref{item:troppath2}.  As this is valid for all $(\bar x, \bar w) \in \tropP$ at
  finite distance of $\log_t(x,w)$, we obtain that $\log_t (x,w)_k \leq (x^{\lambda}, w^{\lambda})_k + \delta(t)$.

  Using similar arguments, we can also show that $\log_t (y,s)_k \leq (y^\lambda, s^\lambda)_k + \delta(t)$, and this
  gives the claimed inequality~\eqref{eq:upper_bound}.   The uniform convergence of $\log_t \cpath_t$ to
  $\troppath$ is now a consequence of the following lemma.
\end{proof}

\begin{lemma}\label{lemma:delta}
  The limit of $\delta(t)$ equals $0$ when $t$ goes to $+\infty$.
\end{lemma}

\begin{proof}
We already argued in the proof of Theorem~\ref{th:uniform} that $\funk\bigl(\log_t (\bm y^*(t), \bm s^*(t)), (y^*, s^*)\bigr) < +\infty$ when $t$ is large enough. Then, the convergence towards $0$ is straightforward.

In order to prove that $\hilbert(\log_t \puiseuxP(t), \tropP)$ tends to $0$, we use the fact that the Hardy polyhedron $\puiseuxP$ can be written as the Minkowski sum of the convex hull (over $\K$) of a finite set of points $\{\bm p^k\}_{k \in K}$ and a polyhedral cone generated by a finite set of rays $\{\bm r^l\}_{l \in L}$, as in the proof of Proposition~\ref{prop-direct}. As shown there, the set $\tropP$ is generated by the points $p^k := \val(\bm p^k)$, ($k \in K$) and rays $r^l := \val(\bm r^l)$ ($l \in L$). Besides, if $t$ is large enough, then the real polyhedron $\puiseuxP(t)$ is generated by the points $\bm p^k(t)$ and rays $\bm r^l(t)$. Provided that $t$ is large enough, $\bm p^k_h(t) = 0$ is equivalent to $p^k_h = -\infty$, for all $k \in K$ and $h \in [m+n]$. Thus, $\funk(\log_t \bm p^k(t), p^k)$ and $\funk(p^k, \log_t \bm p^k(t))$ converge to $0$ when $t \to +\infty$. Similar properties apply to the vectors~$\bm r^l(t)$ and $r^l$. 

Now consider $z := (x,w) \in \puiseuxP(t)$. Let $(\alpha_k)_{k \in K}$ and $(\beta_l)_{l \in L}$ such that $\alpha_k, \beta_l \in \R_{> 0}$ for all $k \in K$, $l \in L$, $\sum_{k \in K} \alpha_k = 1$, and $z = \sum_{k \in K} \alpha_k \bm p^k(t) + \sum_{l \in L} \beta_l \bm r^l(t)$. Then, for all $h \in [m+n]$, we can write
\begin{multline}
\tsum_{k \in K} \bigl((\log_t \alpha_k) \tdot \log_t \bm p^k(t)\bigr) \tplus \tsum_{l \in L} \bigl((\log_t \beta_l) \tdot \log_t \bm r^l(t)\bigr) \leq \log_t z_h \\
\leq \biggl[\tsum_{k \in K} \bigl((\log_t \alpha_k) \tdot \log_t \bm p^k(t)\bigr) \tplus \tsum_{l \in L} \bigl((\log_t \beta_l) \tdot \log_t \bm r^l(t)\bigr)\biggr] + \log_t (|K| + |L|) \, .
\label{eq:delta_proof_eq1}
\end{multline}
Setting $\gamma := \max_{k \in K} \alpha_k$, we have $\frac{1}{|K|} \leq \gamma \leq 1$. Then, we define $z' := \bigl(\tsum_{k \in K} \alpha'_k \tdot p^k\bigr) \tplus \bigl(\tsum_{l \in L} \beta'_l \tdot r^l\bigr)$, where $\alpha'_k := \log_t (\alpha_k/\gamma)$ and $\beta'_l := \log_t \beta_l$. As $\tsum_{k \in K} \alpha'_k = 0$ by definition of $\gamma$, we have $z' \in \tropP$. Besides, $z_h > 0$ if, and only if, there exists $k \in K$ such that $\bm p^k_h(t) > 0$ or $l \in L$ such that $\bm r^l_h(t) > 0$. Provided that $t$ is sufficiently large, this is equivalent to the fact that $p^k_h > -\infty$ for some $k \in L$, or $r^l_h > -\infty$ for a certain $l \in L$. This latter property amounts to $z'_h > -\infty$. Consequently, we have $\hilbert(\log_t z, z') < +\infty$, and we can derive from~\eqref{eq:delta_proof_eq1} that
\begin{multline*}
z'_h - \max\bigl(\log_t |K|+ \max_{k \in K} \funk(\log_t \bm p^k(t), p^k), \max_{l \in L} \funk(\log_t \bm r^l(t), r^l)\bigr) \leq \log_t z_h \\
 \leq z'_h + \log_t (|K| + |L|) + \max\bigl(\max_{k \in K} \funk(p^k, \log_t \bm p^k(t)), \max_{l \in L} \funk(r^l, \log_t \bm r^l(t))\bigr) \, ,
\end{multline*}
for all $h \in [m+n]$. We deduce that 
\begin{multline*}
\hilbert(\log_t \puiseuxP(t), \tropP) \leq \log_t (|K| + |L|) + \max\bigl(\max_{k \in K} \funk(p^k, \log_t \bm p^k(t)), \max_{l \in L} \funk(r^l, \log_t \bm r^l(t))\bigr) \\
+ \max\bigl(\log_t |K|+ \max_{k \in K} \funk(\log_t \bm p^k(t), p^k), \max_{l \in L} \funk(\log_t \bm r^l(t), r^l)\bigr) \, , 	
\end{multline*}
which tends to $0$ when $t \to +\infty$. 

A similar argument works for $\puiseuxQ(t)$ and $\tropQ$, and thus we can conclude that the limit of $\delta(t)$ is equal to~$0$.
\end{proof}

Finally, we note that the uniform convergence of the maps $(\log_t \cpath_t)_t$ to $\troppath$ allows us to recover the result of Corollary~\ref{cor:geometric_characterization} and the duality property~\eqref{eq:primal_dual}. 

\subsection{The tropical central path can degenerate to a simplex path}\label{sec:simplex}

In this section, we restrict our attention to the $x$-component of the primal tropical central path. We show that under
some assumptions, this projection of the tropical central path lies on the image under the valuation map of the graph of
the polyhedron $\puiseuxR := \{ \bm x \in \K^n \mid \bm A \bm x \leq \bm b \, , \, \bm x \geq 0 \}$, \ie~the projection
of the feasible set $\puiseuxP$ onto the $\bm x$-subspace.

We consider the following primal linear program over $\K$:
\begin{equation}\tag*{$\textbf{LP}$}\label{eq:lpl}
\begin{array}{r@{\quad}l}
\text{minimize} & \bm x_n \\
\text{subject to} & \bm x \in \puiseuxP \, ,
\end{array}
\end{equation}
and we make the following assumption:
\begin{assumption}\label{ass:simplex}
\begin{enumerate}[(i)]
\item\label{item:simplex0} The Hardy polyhedron $\puiseuxP$ is bounded, and it contains a point with positive entries;
\item\label{item:simplex1} the optimal value of~\ref{eq:lpl} is $0$;
\item\label{item:simplex3} the tropicalization of the extended matrix $(\bm A \ \bm b)$ is sign generic;
\item\label{item:simplex2} the matrix $\bm A$ contains at most one positive entry in every row.
\end{enumerate}
\end{assumption}
Condition~\eqref{item:simplex0} ensures Assumption~\ref{ass:hardy_central_path_defined} is satisfied (see
Remark~\ref{remark:equivalent_assumption}), so that the results of the previous sections apply. Thanks to
Condition~\eqref{item:simplex1} and the fact that the cost vector in the linear program~\ref{eq:lpl} is non-negative, we
know by Corollary~\ref{cor:geometric_characterization} that the primal part $(x^\lambda, w^\lambda)$ of the tropical
central path is given by the barycenter of the tropical polyhedron $\tropP \cap \{ (x, w) \in \trop^{m+n} \mid x_n \leq
\lambda \}$. As the tropical polyhedron $\tropR := \val(\puiseuxR)$ is the projection of $\tropP$ on the $x$-component,
we deduce that the point $x^\lambda$ corresponds to the barycenter of the tropical polyhedron $\tropR \cap \{ x \in
\trop^n \mid x_n \leq \lambda \}$.

In view of Assumption~\ref{ass:simplex}~\eqref{item:simplex3} we may apply Theorem~\ref{thm:val_inter_commutes} to
obtain a description of $\tropR$ in terms of tropical halfspaces. More precisely, for $A^+ = \val(\bm A^+)$, $A^- =
\val(\bm A^-)$, $b^+ = \val(\bm b^+)$ and $b^- = \val(\bm b^-)$ as usual, we have
\[
\tropR = \bigl\{ x \in \trop^n \mid A^+ \tdot x \tplus b^- \leq A^- \tdot x \tplus b^+ \bigr\} \, .
\]

While the first three conditions in the Assumption~\ref{ass:simplex} are standard requirements concerning the general
position, the final property~\eqref{item:simplex2} is very special.  It forces that each tropical halfspace in the
tropicalization is the complement of a single sector; see \cite{Tropical+halfspaces} for details on the combinatorics of
tropical halfspaces.

\begin{proposition}\label{prop:simplex}
  Under Assumption~\ref{ass:simplex}, the $x$-component of the tropical central path associated with the linear
  program~\ref{eq:lpl} is contained in the image under the valuation map of the vertex-edge graph of the Hardy
  polyhedron $\puiseuxR$.
\end{proposition}
\begin{proof}
Let $j \in [n-1]$. We point out that there must exist an index $i_j \in [m]$ such that $A^+_{i_j j} > -\infty$ and 
\[
A^+_{i_j j} \tdot x^\lambda_j = A^-_{i_j} \tdot x^\lambda \tplus b^+_i \, .
\]
If not, we could define the point $x \in \trop^n$ by $x_k = x^\lambda_k$ if $k \neq j$, and $x_j = x^\lambda_j + \epsilon$ where $\epsilon > 0$, and observe that $x$ still satisfies the inequalities $A^+ \tdot x \tplus b^- \leq A^- \tdot x \tplus b^+$ and $x_n \leq \lambda$, provided that $\epsilon$ is small enough. This would contradict the fact that $x^\lambda$ is the tropical barycenter of $\tropR \cap \{ x \in \trop^n \mid x_n \leq \lambda \}$.

Let $I$ be the set formed by the $i_j$, for $j \in [n-1]$. By Assumption~\ref{ass:simplex}~\eqref{item:simplex2}, we know that every row of the matrix $A^+$ contains at most a finite entry. We deduce that the indices $i_j$ are pairwise distinct, and so the set $I$ has cardinality $n-1$. Besides, as $A^+_I \tdot x^\lambda \tplus b^-_I = A^-_I \tdot x^\lambda \tplus b^+_I$, the second part of Theorem~\ref{thm:val_inter_commutes} ensures that there exists $\bm x \in \puiseuxR$ with $\val(\bm x) = x^\lambda$ satisfying $\bm A_I \bm x = \bm b_I$. 

We claim that the matrix $\bm A_I$ has rank $n-1$. To see this, consider the submatrix $\bm A'$ of $\bm A_I$ formed by
the columns of index $j \in [n-1]$. The expansion of its determinant contains the term $\pm \prod_{j \in [n-1]} \bm
A_{i_j j}$. The valuation of this term is given by $\bigodot_{j \neq l} A^+_{i_j j} > -\infty$, hence the term cannot be
null. Since the tropicalization of $\bm A'$ is sign generic by Assumption~\ref{ass:simplex}~\eqref{item:simplex3}, we
deduce that $\det \bm A' \neq 0$, which proves the claim.

As a consequence, the point $\bm x$ satisfies at least $n-1$ linearly independent defining inequalities of $\puiseuxR$
with equality.  Therefore, $\bm x$ belongs to the vertex-edge graph of $\puiseuxR$, and so $x^\lambda$ lies in its image
under the valuation map.
\end{proof}

Notice that the Assumption~\ref{ass:simplex} is sufficient but not necessary for the tropical central path to degenerate
to the boundary.  In Example~\ref{ex:ex1} the first inequality in (\ref{eq:ex}) has two positive coefficients, and still
the tropical central path lies in the boundary; see Figure~\ref{fig:path1} (left).

\begin{remark}
  The analytic center of $\puiseuxR$, defined as the projection on the $\bm x$-component of the analytic center of
  $\puiseuxP$, is sent by the valuation map to the tropical barycenter of $\tropR$ (by
  Corollary~\ref{coro:analytic_center}). The latter point can be shown to coincide with the value of a vertex of
  $\puiseuxR$, using the arguments of the proof of Proposition~\ref{prop:simplex}. It suffices to observe that an index
  $i_j$ can be found for all $j \in [n]$, which provides a subsystem of $n$ linearly independent inequalities.
\end{remark}

\section{Long tropical central paths and ordinary central paths with high curvature}
\noindent
Bezem, Nieuwenhuis and Rodr{\'{\i}}guez-Carbonell~\cite{BezemNieuwenhuisRodriguez08} constructed a class of tropical
linear equalities for which an algorithm of Butkovi\v{c} and Zimmermann~\cite{Butkovic2006} exhibits an exponential
running time. This gives rise to tropical linear programs which we lift to the Hardy field $ \K =H(\mystruct)$. From this, we obtain a one-parameter family of ordinary linear programs over the reals. The latter are interesting as their central
paths have an unusually high total curvature, as we shall see in Section~\ref{subsec:curvature}.

\begin{figure}[t]
  \resizebox{!}{.4\textwidth}{ 

\begin{tikzpicture}[x  = {(0.788528674827569cm,0.573568948029112cm)},
                    y  = {(0.740156303431162cm,-0.152008463919665cm)},
                    z  = {(0.00777409926220149cm,0.450108449723558cm)},
                    scale = 1,
                    color = {lightgray}]

  \definecolor{pointcolor_p1}{rgb}{ 1,0,0 }
  \tikzstyle{pointstyle_p1} = [fill=pointcolor_p1]

  \coordinate (v0) at (1.14286, 1.62314, 3.8299);
  \coordinate (v1) at (1.14286, 1.62314, 0.0818038);
  \coordinate (v2) at (2, 0, 2.82843);
  \coordinate (v3) at (2, 0, 0);
  \coordinate (v4) at (0, 0, 0);
  \coordinate (v5) at (0, 4, 0);
  \coordinate (v6) at (0, 4, 5.65685);
  \coordinate (v7) at (2, 4, 0);
  \coordinate (v8) at (1.14286, 2.94829, 0.0818038);
  \coordinate (v9) at (2, 4, 8.48528);
  \coordinate (v10) at (1.14286, 2.94829, 5.70395);

  \definecolor{linecolor_p1}{rgb}{ 0.4666666667,0.9254901961,0.6196078431 }
  \tikzstyle{linestyle_p1} = [color=linecolor_p1, thick]

  \definecolor{linecolor_p1_v1_v0}{rgb}{ 0.4667,0.9255,0.6196 }

  \definecolor{linecolor_p1_v2_v0}{rgb}{ 0.4667,0.9255,0.6196 }

  \definecolor{linecolor_p1_v3_v1}{rgb}{ 0.4667,0.9255,0.6196 }

  \definecolor{linecolor_p1_v3_v2}{rgb}{ 0,0,0 }

  \definecolor{linecolor_p1_v4_v0}{rgb}{ 0.4667,0.9255,0.6196 }

  \definecolor{linecolor_p1_v4_v1}{rgb}{ 0.4667,0.9255,0.6196 }

  \definecolor{linecolor_p1_v4_v2}{rgb}{ 0,0,0 }

  \definecolor{linecolor_p1_v4_v3}{rgb}{ 0,0,0 }

  \definecolor{linecolor_p1_v5_v4}{rgb}{ 0,0,0 }

  \definecolor{linecolor_p1_v6_v4}{rgb}{ 0,0,0 }

  \definecolor{linecolor_p1_v6_v5}{rgb}{ 0,0,0 }

  \definecolor{linecolor_p1_v7_v3}{rgb}{ 0,0,0 }

  \definecolor{linecolor_p1_v7_v5}{rgb}{ 0,0,0 }

  \definecolor{linecolor_p1_v8_v1}{rgb}{ 0.4667,0.9255,0.6196 }

  \definecolor{linecolor_p1_v8_v5}{rgb}{ 0.4667,0.9255,0.6196 }

  \definecolor{linecolor_p1_v8_v7}{rgb}{ 0.4667,0.9255,0.6196 }

  \definecolor{linecolor_p1_v9_v2}{rgb}{ 0,0,0 }

  \definecolor{linecolor_p1_v9_v6}{rgb}{ 0,0,0 }

  \definecolor{linecolor_p1_v9_v7}{rgb}{ 0,0,0 }

  \definecolor{linecolor_p1_v10_v0}{rgb}{ 0.4667,0.9255,0.6196 }

  \definecolor{linecolor_p1_v10_v6}{rgb}{ 0.4667,0.9255,0.6196 }

  \definecolor{linecolor_p1_v10_v8}{rgb}{ 0.4667,0.9255,0.6196 }

  \definecolor{linecolor_p1_v10_v9}{rgb}{ 0.4667,0.9255,0.6196 }

  \draw[linestyle_p1, linecolor_p1_v3_v2, dashed] (v3) -- (v2);
  \draw[linestyle_p1, linecolor_p1_v4_v3, dashed] (v4) -- (v3);
  \draw[linestyle_p1, linecolor_p1_v7_v3, dashed] (v7) -- (v3);
  \draw[linestyle_p1, linecolor_p1_v1_v0] (v1) -- (v0);
  \draw[linestyle_p1, linecolor_p1_v2_v0] (v2) -- (v0);
  \draw[linestyle_p1, linecolor_p1_v3_v1] (v3) -- (v1);
  \draw[linestyle_p1, linecolor_p1_v4_v0] (v4) -- (v0);
  \draw[linestyle_p1, linecolor_p1_v4_v1] (v4) -- (v1);
  \draw[linestyle_p1, linecolor_p1_v8_v1] (v8) -- (v1);
  \draw[linestyle_p1, linecolor_p1_v8_v5] (v8) -- (v5);
  \draw[linestyle_p1, linecolor_p1_v8_v7] (v8) -- (v7);
  \draw[linestyle_p1, linecolor_p1_v10_v0] (v10) -- (v0);
  \draw[linestyle_p1, linecolor_p1_v10_v6] (v10) -- (v6);
  \draw[linestyle_p1, linecolor_p1_v10_v8] (v10) -- (v8);
  \draw[linestyle_p1, linecolor_p1_v10_v9] (v10) -- (v9);

  \draw[linestyle_p1, linecolor_p1_v4_v2] (v4) -- (v2);
  \draw[linestyle_p1, linecolor_p1_v5_v4] (v5) -- (v4);
  \draw[linestyle_p1, linecolor_p1_v6_v4] (v6) -- (v4);
  \draw[linestyle_p1, linecolor_p1_v6_v5] (v6) -- (v5);
  \draw[linestyle_p1, linecolor_p1_v7_v5] (v7) -- (v5);
  \draw[linestyle_p1, linecolor_p1_v9_v2] (v9) -- (v2);
  \draw[linestyle_p1, linecolor_p1_v9_v6] (v9) -- (v6);
  \draw[linestyle_p1, linecolor_p1_v9_v7] (v9) -- (v7);

  \node at (v4) [inner sep=0.5pt, below left, black] {\tiny{$(0,0,0,0)$}};
  \fill[pointcolor_p1] (v4) circle (1 pt);

  \node at (v3) [inner sep=0.5pt, above left, black] {\colorbox{white}{\tiny{$(t,0,0,0)$}}};
  \fill[pointcolor_p1] (v3) circle (1 pt);

  \node at (v1) [rotate=45, inner sep=0.5pt, below left, black] {\colorbox{white}{\tiny{$(t,t,t^2,0)$}}};
  \fill[pointcolor_p1] (v1) circle (1 pt);

  \node at (v5) [inner sep=0.5pt, below right, black] {\tiny{$(0,t^2,0,0)$}};
  \fill[pointcolor_p1] (v5) circle (1 pt);

  \node at (v2) [inner sep=0.5pt, above left, black] {\tiny{$(t,0,0,t^{3/2})$}};
  \fill[pointcolor_p1] (v2) circle (1 pt);

  \node at (v7) [inner sep=0.5pt, above right, black] {\tiny{$(t,t^2,0,0)$}};
  \fill[pointcolor_p1] (v7) circle (1 pt);

  \node at (v0) [rotate=45, inner sep=0.5pt, above right, black] {\colorbox{white}{\tiny{$(t,t,t^2,2t^{3/2})$}}};
  \fill[pointcolor_p1] (v0) circle (1 pt);

  \node at (v6) [inner sep=0.5pt, below right, black] {\colorbox{white}{\tiny{$(0,t^2,0,t^{3/2})$}}};
  \fill[pointcolor_p1] (v6) circle (1 pt);

  \node at (v8) [inner sep=0.5pt, below right, black] {\colorbox{white}{\tiny{$(t,t^2,t^2,0)$}}};
  \fill[pointcolor_p1] (v8) circle (1 pt);

  \node at (v10) [inner sep=0.5pt, above right, black] {\colorbox{white}{\tiny{$(t,t^2,t^2,t^{5/2}+t^{3/2})$}}};
  \fill[pointcolor_p1] (v10) circle (1 pt);
  \node at (v9) [inner sep=0.5pt, above right, black] {\tiny{$(t,t^2,0,t^{5/2} + t^{3/2})$}};
  \fill[pointcolor_p1] (v9) circle (1 pt);

\end{tikzpicture}
}
  \caption{Schlegel diagram of $\puiseuxR_1$ (and $t\geq2$), projected onto the facet $\bm u_1=0$; the points are written in $(\bm u_0, \bm v_0, \bm u_1, \bm v_1)$-coordinates.} \label{fig:schlegel}
\end{figure}
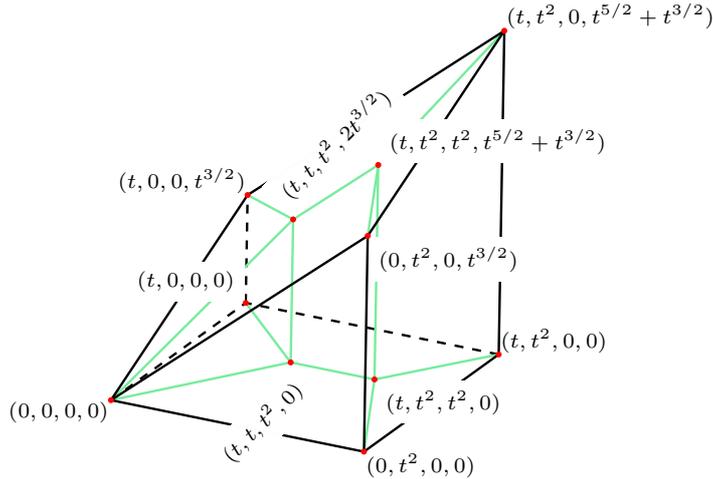

Given a positive integer $r$, we introduce the following linear program over the Hardy field $\K$ in the $2r+2$
variables $\bm u_0, \bm v_0, \bm u_1, \bm v_1, \dots, \bm u_r, \bm v_r$:
\[
\begin{array}{r@{\quad}l}
\text{minimize} & \, \bm v_0 \\[1ex]
\text{subject to} 
&
\begin{aligned}[t]
\bm u_0 & \leq t \, , \ \bm v_0 \leq t^2 \\
& \left.\mybackup{\bm u_i}\begin{aligned}
\bm u_i & \leq t \bm u_{i-1} \, , \ \bm u_i \leq t\bm v_{i-1} \vphantom{t^{1 - \frac{1}{2}}} \\
\bm v_i & \leq t^{1 - \frac{1}{2^{i}}} (\bm u_{i-1} + \bm v_{i-1}) 
\end{aligned}\; \right\} \qquad \text{for} \ 1\leq i\leq r\\
\bm u_r & \geq 0 \, , \  \bm v_r \geq 0 \, . \vphantom{t^{1 - \frac{1}{2}}}  
\end{aligned}
\end{array}
\]
The optimal value of this linear program equals $0$, and an optimal solution is given by $\bm u = \bm v = 0$.  Moreover,
the feasible set $\puiseuxR_r$ is a polytope contained in the positive orthant, and the $3r+4$ inequalities listed
define its facets. In particular, the remaining non-negativity constraints $\bm u_i\ge0$ and $\bm v_i\ge 0$ for $0 \leq
i < r$ are satisfied but redundant.  For sufficiently large real $t$ each minor of the real constraint matrix
stabilizes, and thus the chirotope stabilizes, too.  It follows that the combinatorial type of the real polytopes
$\puiseuxR_r(t)$ stabilizes for $t$ sufficiently large, and this combinatorial type coincides with the combinatorial type of
the Hardy polytope $\puiseuxR_r$.  Figure~\ref{fig:schlegel} shows an example for $r=1$ and $t\geq2$, which \emph{is}
large enough in this case.

In our subsequent analysis we will work with the equivalent linear program expressed with the $3r+2$ slack variables
$\bm h_i$, $\bm z_i$, $\bm z_i'$, for $0\leq i\leq r$, but $\bm z_0'$ does not occur:
\begin{equation} \tag*{$\hardyex_r$} \label{eq:cex}
\begin{array}{r@{\quad}l}
\text{minimize} & \, \bm v_0 \\[1.25ex]
\text{subject to} & 
\begin{aligned}[t]
\bm u_0 + \bm z_{0} & = t \\
\bm v_0 + \bm h_{0} & = t^2 \vphantom{t^{1 - \frac{1}{2}}} \\
&
\left.\mybackup{\bm u_i + \bm z'_i}\begin{aligned}
\bm u_i +\bm z_{i} & = t \bm u_{i-1} \vphantom{t^{1 - \frac{1}{2}}} \\
\bm u_i +\bm z'_{i} & =  t\bm v_{i-1} \vphantom{t^{1 - \frac{1}{2}}} \\
\bm v_i + \bm h_{i} & = t^{1 - \frac{1}{2^{i}}} (\bm u_{i-1} + \bm v_{i-1}) 
\end{aligned}\; \right\} && \text{for} \ 1\leq i\leq r\\[1ex]
& \left.\mybackup{\bm u_i + \bm z'_i}\begin{aligned}[b]
\bm u_i & \geq 0 \, , & \bm v_i & \geq 0 \, , & \bm z_i & \geq 0 \, , & \bm h_i & \geq 0 
\end{aligned}\right. && \text{for} \ 0 \leq i\leq r \\
& \left.\mybackup{\bm u_i + \bm z'_i}\begin{aligned}[b]
&\bm z_i' \, \geq 0 
\end{aligned}\right. && \text{for} \ 1 \leq i\leq r \, .
\end{aligned}
\end{array}
\end{equation}
We denote by $\puiseuxP_r\subset\K^{5r+4}$ the feasible set of~\ref{eq:cex}. Assumption~\ref{ass:hardy_central_path_defined} is satisfied because $\puiseuxP_r$ is a polytope, and we can find a strictly feasible point by setting, \eg, $\bm u_0=\bm z_0=\frac{1}{2}t$, $\bm v_0=\bm h_0=\frac{1}{2}t^2$ and, inductively,
\[
\bm u_i=\bm z_i=\frac{1}{2}t\bm u_{i-1} \,, \quad \bm z_i'=t\bm v_{i-1}-\bm u_i \,, \quad \bm v_i=\bm
h_i=\frac{1}{2}t^{1 - \frac{1}{2^{i}}} (\bm u_{i-1} + \bm v_{i-1}) \, .
\]

\subsection{Computing the primal tropical central path}

We will determine the tropical central path arising from the linear program~\ref{eq:cex} using the characterization
established in Section~\ref{sec:central_path}. We focus on the primal component, as the dual can be easily obtained from
the former by the relations~\eqref{eq:primal_dual}. We denote by $\troppathprimal(\lambda)$ the point on the primal
tropical central path associated with the parameter $\lambda \in \R$. By Corollary~\ref{cor:geometric_characterization},
that point corresponds to the barycenter of the tropical polyhedron $\val(\puiseuxP_r)$ intersected with the
sublevel set $v_0 \leq \lambda$. 

The formal tropicalization of the objective function and the constraints of \ref{eq:cex} yields the tropical linear
program
\begin{equation} \tag*{$\tropex_r$} \label{eq:tex}
\begin{array}{r@{\quad}l}
\text{minimize} & \, v_0 \\[1.25ex]
\text{subject to} & 
\begin{aligned}[t]
\max(u_0,z_0) & = 1 \\
\max(v_0,h_0) & = 2 \\
& \left.\mybackup{\max(v_i, h_i)}\begin{aligned}
\max(u_i,z_i) & = 1+ u_{i-1} \\ 
\max(u_i,z'_i) &  = 1+ v_{i-1} \\
\max(v_i, h_i) & = 1 - \frac{1}{2^{i}} + \max (u_{i-1}, v_{i-1}) 
\end{aligned}\; \right\} \qquad \text{for} \ 1 \leq i \leq r \, .
\end{aligned}
\end{array}
\end{equation}
For any real $\lambda$ we call the intersection of the feasible region of \ref{eq:tex} with the tropical halfspace $\{
(u,v,z,z',h) \in \trop^{5r+4} \mid v_0 \leq \lambda\}$ the \emph{$\lambda$-sublevel set} of \ref{eq:tex}.

If we would verify that the extended matrix of \ref{eq:cex} is tropically sign generic, then we could apply
Theorem~\ref{thm:val_inter_commutes} to conclude that the feasible region of \ref{eq:tex} is precisely the image of the
feasible region of \ref{eq:cex} under the valuation map.  Moreover, we could then apply
Corollary~\ref{cor:geometric_characterization} to see that the tropical central path is given by the tropical barycenter
of the $\lambda$-sublevel set.  However, we prefer to avoid the somewhat tedious verification of tropical sign
genericity. Instead, we first compute in Lemma~\ref{lem-dynamic} the tropical barycenter of the $\lambda$-sublevel set
of \ref{eq:tex}, and show that it coincides with the orbit of a piecewise linear dynamical system. Then, we will show
directly, by an elementary argument, that this orbit is the tropical central path, see Proposition~\ref{prop:lift}.  The
dynamical system involves the following family of transition maps
\[
G_i(a,b) := \Bigl(1+\min(a,b) \, , \, 1-\frac{1}{2^i} + \max(a,b)\Bigr) \, ,
\qquad 1\leq i \leq r
\]
from $\R^2$ to itself.  

\begin{lemma}\label{lem-dynamic}
  For each $\lambda\in \R$, the point $(u(\lambda),v(\lambda),z(\lambda),z'(\lambda),h(\lambda))$ is the tropical
  barycenter of the $\lambda$-sublevel set of\/ \ref{eq:tex}, where:
\begin{subequations}
  \begin{align}
    \qquad\qquad\qquad & \left.\mybackup{(u_i(\lambda), v_i(\lambda))}\begin{aligned}
        u_0(\lambda) & = 1 \, ,\quad v_0 (\lambda)= \min(2,\lambda) \\ 
        (u_i(\lambda) , v_i(\lambda)) & = G_i(u_{i-1}(\lambda),v_{i-1}(\lambda)) \qquad \text{for} \ 1\leq i\leq r \,, 
      \end{aligned}\right.\label{e-dynamics}
    \\
    & \left.\mybackup{z_0(\lambda)}\begin{aligned}
        z_0(\lambda) & =1 \,,  \quad h_0(\lambda)=2 \\
        z_i(\lambda) & =1+u_{i-1}(\lambda) \,, \quad z'_i(\lambda)=1+v_{i-1}(\lambda) \,,\quad
        h_i(\lambda) = v_i(\lambda) \qquad \text{for} \ 1\leq i\leq r \, .
      \end{aligned}\right.\label{e-dynamics2}
  \end{align}
\end{subequations}
\end{lemma}

\begin{proof}
First observe that every point of the $\lambda$-sublevel set satisfies the inequalities:
  \begin{equation}\label{e-sat}
    \begin{gathered}
      u_0 \leq 1 \, , \  v_0 \leq \min(2,\lambda) \,, \  (u_i,v_i) \leq G_i(u_{i-1},v_{i-1}) 
      \\
      z_0\leq 1 \,, \  h_0 \leq 2 \,, \  z_i \leq 1+ u_{i-1} \,,\ 
      z'_i \leq  1+v_{i-1} \,,\   h_i \leq 1 - \frac{1}{2^i} + \max(u_{i-1},v_{i-1})
      \, ,
    \end{gathered}
  \end{equation}
  where $1 \leq i \leq r$. Since the maps $G_i$ are order preserving, the barycenter of the tropical polyhedron defined by~\eqref{e-sat} is the point obtained by attaining equality in~\eqref{e-sat}. Therefore, it is given by the point $(u(\lambda),v(\lambda),z(\lambda),z'(\lambda),h(\lambda))$. As the latter point belongs to the $\lambda$-sublevel set it must be the tropical barycenter.
\end{proof}

We can now give a complete parameterization of the primal tropical central path.

\begin{proposition}\label{prop:lift}
  For each $\lambda\in \R$, the point $\troppathprimal(\lambda)$ on the primal tropical central path coincides with the
  tropical barycenter of the $\lambda$-sublevel set of the tropical linear program~\ref{eq:tex}.
\end{proposition}

\begin{proof}
  The set $\val(\puiseuxP_r)$ is contained in the feasible region of~\ref{eq:tex}.  From
  Corollary~\ref{cor:geometric_characterization} and the previous lemma we have
  \[
  \troppathprimal(\lambda) \leq (u(\lambda),v(\lambda),z(\lambda),z'(\lambda),h(\lambda)) \, .
  \]
  To show that the reverse inequality holds, using Corollary~\ref{cor:geometric_characterization} again, it suffices to
  lift that point to an element of $\puiseuxP_r$.  Let us fix a sequence of positive numbers $\alpha_0
  =\frac{1}{2}>\alpha_1>\dots>\alpha_r >0$.  We claim that
  \begin{align*}
    \bm u_i = \alpha_i t^{u_i(\lambda)} \, , \qquad
    \bm v_i = \alpha_i t^{v_i(\lambda)} \, , \qquad 0\leq i\leq r \, ,
  \end{align*}
  yields such an admissible lift. 
  Indeed, 
  $\bm z_0 := \alpha_0 t$ and 
  $\bm h_0 := t^2 - \alpha_0 t^{v_0(\lambda)}$
  satisfy the two first constraints in $\hardyex_r$, and 
  they are such that $\val \bm z_0 = z_0(\lambda)$ and $\val \bm h_0 =h_0(\lambda)$.
  Moreover, for $1\leq i\leq r$,
  \[
  \bm h_i:= 
  t^{1 - \frac{1}{2^{i}}} (\bm u_{i-1} + \bm v_{i-1}) -\bm v_i 
  = (\epsilon_i(\lambda)\alpha_{i-1}-\alpha_i) t^{v_i(\lambda)}
  +o(t^{v_i(\lambda)})\,,
  \]
  where $\epsilon_i(\lambda) = 1$ if $u_{i-1}(\lambda)\neq v_{i-1}(\lambda)$
  and $\epsilon_i(\lambda) = 2$ otherwise. Since $\epsilon_i(\lambda) \alpha_{i-1}>\alpha_i$, we have $\bm h_i \geq 0$, and $\val \bm h_i = h_i(\lambda)$.
  Similarly, 
  \[ 
  \bm z_{i} := t \bm u_{i-1} 
  - \bm u_i = \alpha_{i-1} t^{1+u_{i-1}(\lambda)} 
  - \alpha_{i} t^{u_i(\lambda)} 
  = (\alpha_{i-1}-\alpha_i)t^{z_i(\lambda)}
  + \alpha_i (t^{z_i(\lambda)}-t^{u_i(\lambda)})
  \,,
  \]
  and since $0\leq t^{z_i(\lambda)}-t^{u_i(\lambda)} \leq t^{z_i(\lambda)}$, we deduce
  that $\bm z_i\geq 0$ and $\val \bm z_i=  z_i(\lambda)$.
  Finally, a similar argument shows that 
  \(
  \bm z'_{i}:=  t\bm v_{i-1}
  - \bm u_i \) satisfies
  $\val \bm z'_i\geq 0$ and 
  $\val \bm z'_i =z_i(\lambda)$.
\end{proof}

We now focus on the $(u,v)$-component of the primal tropical central path, since the $(z,z',h)$-component can be easily
determined using~\eqref{e-dynamics2}.
According to~\eqref{e-dynamics}, the coordinate $v_i(\lambda)$ is equal to the maximum of $u_{i-1} (\lambda)$ and $ v_{i-1}(\lambda) $
translated by $1 - \frac{1}{2^{i}}$, while $u_i(\lambda)$ follows the minimum of these two variables shifted by $1$, see
Figure~\ref{fig:trop_counter_ex} for an illustration.  Since the translation offsets differ by $\frac{1}{2^{i}}$, the
components $u_i$ and $v_i$ cross each other $\Omega(2^{i})$ times.  More precisely, our next result shows that the curve
$(u_i(\lambda), v_i(\lambda))$ has the shape of a staircase with $\Omega(2^{i})$ steps.

\begin{figure}[t]
  \centering
  \begin{tikzpicture}[scale =3,
    curve/.style={black, thick}
]
 \draw[help lines, gray!40!] (0,0) grid (2,5); 
 \foreach \x in {0,...,2} { \node [anchor=north] at (\x,0) {\x}; }
 \foreach \y in {0,...,5} { \node [anchor=east] at (0,\y) {\y}; }

\draw[gray!40!, ->] (0,0) -- (2.5,0);
\node[anchor = north east] at (2.5,0) {$\lambda$};

    \foreach \n in {1,2,3,4} {
      \pgfmathsetmacro{\range}{pow(2,\n-1)-1}
      \foreach \k in {0,...,\range} {
          
 \draw[curve] ( 4 * \k / 2^\n   , \n + 2 * \k / 2^\n + 1 / 2^\n) 
 --  ( 4 * \k  / 2^\n  + 2 / 2^\n , \n + 2 * \k / 2^\n + 1 / 2^\n)
--  ( 4 * \k  / 2^\n  + 4 / 2^\n, \n + 2 * \k / 2^\n  + 3 / 2^\n)
;
 \draw[curve, red] ( 4 * \k / 2^\n   , \n + 2 * \k / 2^\n) 
 --  ( 4 * \k  / 2^\n  + 2 / 2^\n , \n + 2 * \k / 2^\n + 2 / 2^\n )
--  ( 4 * \k  / 2^\n  + 4 / 2^\n, \n + 2 * \k / 2^\n  + 2 / 2^\n)
;
      } 
\node[anchor = north west] at (2,\n + 1) {$u_{\n}$};
\node[anchor = south west] at (2,\n + 1 + 1/2^\n) {$v_{\n}$};
\draw[curve,red] (2, \n+1) -- (2.1, \n+1);
\draw[curve] (2, \n + 1 + 1/2^\n) -- (2.1, \n + 1 + 1/2^\n);

\draw[curve,red] (0, \n) -- (-.1, \n - .1);
\draw[curve] (0, \n  + 1/2^\n) -- (-.1, \n  + 1/2^\n);
    }

  \end{tikzpicture}

  \caption{The $(u,v)$-components of the primal tropical central path of $\hardyex_4$ when $0 \leq \lambda \leq 2$.}
  \label{fig:trop_counter_ex}
\end{figure}
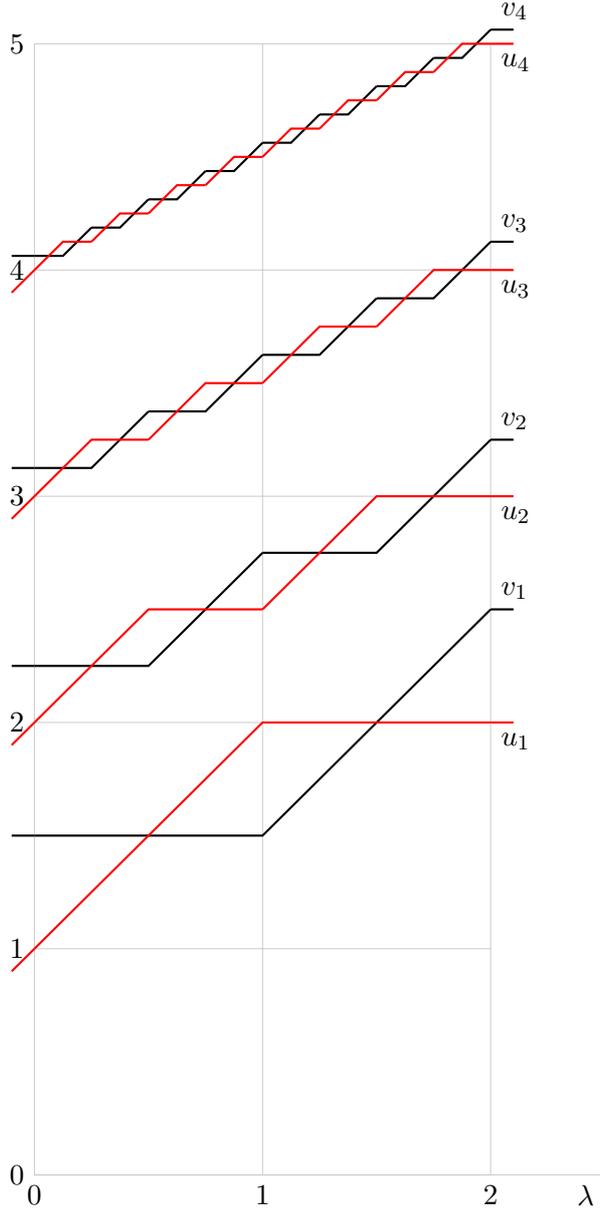

\begin{proposition}\label{prop:counter_ex}
  Let $i\in[r]$ and $k\in\{0, \dots , 2^{i-1} -1\}$.  Then, for all $\lambda \in [\tfrac{4k}{2^{i}}, \tfrac{4k+2}{2^{i}}]$, we have:
\begin{alignat*}{2}
u_{i}(\lambda) & = i + \lambda -\frac{2k}{2^{i}} \quad \text{and} \quad & v_{i}(\lambda) & = i +\frac{2k+1}{2^{i}} \,,
\shortintertext{while for all $\lambda \in [\tfrac{4k+2}{2^{i}}, \tfrac{4k+4}{2^{i}}]$, we have:}
u_{i} (\lambda) & = i + \frac{2k+2}{2^{i}} \quad \text{and} \quad & v_{i} (\lambda) & = i + \lambda - \frac{2k+1}{2^{i}} \, .
\end{alignat*}
\end{proposition}

\begin{proof}
Let us fix $\lambda \in [0,2]$, and we simply denote the $u_i(\lambda)$ and $v_i(\lambda)$ by $u_i$ and $v_i$ respectively.  We proceed by induction on $i\in[r]$. It follows from~\eqref{e-dynamics} that 
\[
u_1 = 1 + \min(1, \lambda) \, , \qquad v_1 = \frac{1}{2} + \max(1, \lambda)
\, .
\]
Thus for $\lambda \in [0, 1]$, $u_1 = 1+\lambda $ and $v_1 = 1 + \frac{1}{2}$.  For $\lambda
\in [1, 2]$ we have $u_1 = 1 + 1$ and $v_1 = 1 + \lambda - \frac{1}{2} $.  Consequently, the statement holds for $i=1$.

We suppose the result is verified for $i < r$, and show that it holds for $i+1$. Let $k \in \{ 0, \dots , 2^i -1 \}$.
If $k$ is even, let $k'=k/2$.  Then, for all $\lambda$ in the interval $[\frac{4k}{2^{i+1}}, \frac{4k+4}{2^{i+1}}] =
[\frac{4k'}{2^{i}}, \frac{4k'+2}{2^{i}}]$, we have by induction:
\[
u_{i}=i+ \lambda -\frac{2k'}{2^{i}} = i+ \lambda - \frac{k}{2^{i}} \quad \text{and} \quad
v_i= i +\frac{2k'+1}{2^{i}} = i+\frac{k+1}{2^{i}}  \, .
\]
Then,
\[
u_{i+1} =  i+1 + \min\Bigl(\frac{k+1}{2^{i}} ,\lambda - \frac{k}{2^{i}} \Bigr) \quad \text{and} \quad
v_{i+1} =  i+1 + \max \Bigl(  \frac{k+1}{2^{i}}, \lambda- \frac{k}{2^{i}} \Bigr) -\frac{1}{2^{i+1}} \, .
\]
Separating the cases $\lambda \leq \frac{4k+2}{2^{i+1}}$ and $\lambda \geq \frac{4k+2}{2^{i+1}}$ leads to the expected result.

If $k$ is odd, $k=2k'+1$, then for any $\lambda \in [\frac{4k}{2^{i+1}}, \frac{4k+4}{2^{i+1}}] = [\frac{4k'+2}{2^{i}},
\frac{4k'+4}{2^{i}}]$ we have:
\[
u_{i} = i +\frac{2k'+2}{2^{i}}  = i +\frac{k+1}{2^{i}} \quad \text{and} \quad
v_{i} = i + \lambda - \frac{2k'+1}{2^{i}} = i + \lambda- \frac{k}{2^{i}} \, .
\]
Thus,
\[
u_{i+1} = i+1 + \min \Bigl(\lambda -  \frac{k}{2^{i}}, \frac{k+1}{2^{i}} \Bigr) \quad \text{and} \quad v_{i+1} = i+1  + \max \Bigl(\lambda -  \frac{k}{2^{i}}, \frac{k+1}{2^{i}} \Bigr) - \frac{1}{2^{i+1}} \, .
\]
As above, by separating the cases $\lambda \leq \frac{4k+2}{2^{i+1}}$ and $\lambda \geq \frac{4k+2}{2^{i+1}}$, we conclude that the statement holds for $i+1$.
\end{proof}

\begin{remark}\label{rem:modified}
  A similar induction shows that for all $\lambda \geq 2$, the primal tropical central path is constant, equal to the
  tropical barycenter of $\val(\puiseuxR_r)$, whose $(u,v)$-component is defined by $u_0 = 1$, $v_0 =2$, and:
  \[
  u_{i}(\lambda) = i + 1 \quad \text{and} \quad  v_{i}(\lambda) = i + 1 + \frac{1}{2^{i}} \qquad \text{for all} \  1\leq i\leq r \, .
  \]
  For $\lambda \leq 0$, the primal tropical central path consists of a half-line towards an optimal solution
  of~\ref{eq:cex}. We have $u_0(\lambda) = 1$, $v_0(\lambda) = \lambda$ as well as:
  \[
  u_{i} (\lambda)= i + \lambda \quad \text{and} \quad v_{i} (\lambda)= i + \frac{1}{2^{i}}  \qquad \text{for all } 1\leq i\leq r \, .
  \]
\end{remark}

Table~\ref{tab:table} gives a summary of relevant coordinate values related to the primal tropical central path.

\begin{table}[tb]
\caption{Coordinates of points on the primal tropical central path of \ref{eq:cex} for some specific values of $\lambda$,
  where $1 \leq i \leq r$ and $k = 0, 2, \dots, 2^{i-1}-2$.}
\label{tab:table}
\renewcommand{\arraystretch}{1.5}
\begin{tabular*}{.75\linewidth}{@{\extracolsep{\fill}}>{$}c<{$}>{$}c<{$}>{$}c<{$}>{$}c<{$}>{$}c<{$}>{$}c<{$}@{}}\toprule
 \lambda & \frac{4k}{2^i} & \frac{4k+2}{2^i} & \frac{4k+4}{2^i} & \frac{4k+6}{2^i} & \frac{4k+8}{2^i} \\
\midrule
 u_i & i + \frac{2k}{2^i} & i + \frac{2k+2}{2^i} & i + \frac{2k+2}{2^i} & i + \frac{2k+4}{2^i} & i + \frac{2k+4}{2^i} \\
 v_i & i + \frac{2k+1}{2^i} & i + \frac{2k+1}{2^i} & i + \frac{2k+3}{2^i} & i + \frac{2k+3}{2^i} & i + \frac{2k+5}{2^i} \\
 z_i &  i + \frac{2k}{2^i} & i + \frac{2k+2}{2^i} & i + \frac{2k+4}{2^i} & i + \frac{2k+4}{2^i} & i + \frac{2k+4}{2^i} \\
 z'_i & i + \frac{2k+2}{2^i} & i + \frac{2k+2}{2^i} & i + \frac{2k+2}{2^i} & i + \frac{2k+4}{2^i} & i + \frac{2k+6}{2^i} \\
 h_i & i + \frac{2k+1}{2^i} & i + \frac{2k+1}{2^i} & i + \frac{2k+3}{2^i} & i + \frac{2k+3}{2^i} & i + \frac{2k+5}{2^i} \\
\bottomrule
\end{tabular*}
\end{table}

\subsection{Curvature analysis}\label{subsec:curvature}

The linear program $\hardyex_r$ gives rise to a family of ordinary linear programs $\hardyex_r(t)$ over the reals,
obtained by instantiating the parameter $t$ by a real value. We reuse the notation of Section~\ref{sec:central_path},
and given $\lambda \in \R$, we denote by $\cpath_t(\lambda)$ the point on the central path of the linear program
$\hardyex_r(t)$ with parameter $t^\lambda$, and by $\bm \cpath(\lambda)$ the germs of the functions $t \mapsto
\cpath_t(\lambda)$, see Lemma~\ref{lem:hardy_central_path}. In this section, we exploit the characterization of
the tropical central path, \ie~the map $\val(\bm \cpath(\cdot))$, to establish lower bounds on the curvature of the
classical central paths $\cpath_t$ when $t$ is large enough.

Given two non-null vectors $x, y \in \R^p$, we denote by $\angle xy$ the measure $\alpha \in [0,\pi]$ of the angle of the vectors $x$ and $y$, so that
\[
\cos \alpha = \frac{\scalar{x}{y}}{\norm{x} \norm{y}} \, ,
\]
where $\scalar{\cdot}{\cdot}$ and $\norm{\cdot}$ refer to the Euclidean scalar product and the associated norm respectively. By extension, given three points $U, V, W \in \R^p$ such that $U \neq V$ and $V \neq W$, we denote by $\angle U V W$ the angle formed by the vectors $U V$ and $V W$. 
Recall that if $\tau$ is a polygonal curve in $\R^p$ parameterized over an interval $[a,b]$, the total curvature $\kappa(\tau,[a,b])$ is defined as the sum of angles between the consecutive segments of the curve. More generally, the total curvature $\kappa(\sigma,[a,b])$ can be defined for an arbitrary curve $\sigma$, parameterized over the same interval, as the supremum of $\kappa(\tau,[a,b])$ over all polygonal curves $\tau$ inscribed in $\sigma$. When $\sigma$ is twice continuously differentiable, this coincides with the standard definition of the total curvature $\int_a^b \norm{\kappa''(s)}ds$, when $\kappa$ is parameterized by arc length, see~\cite[Chapter~V]{nonsmoothpaths} for more background. 

Our approach rests on estimating the curvature of $\cpath_t$ using approximations by polygonal curves. To this end, we prove the following lemma:
\begin{lemma}\label{lemma:angle}
Let $\bm x$, $\bm y$ be two non-null vectors in $\K^p$, and let $x := \val(\bm x)$ and $y := \val(\bm y)$. The limit of $\angle\bm x(t) \bm y(t)$ exists, and if the sets $\argmax_{i \in [p]} x_i$ and $\argmax_{i \in [p]} y_i$ are disjoint, then 
\[
\lim_{t \to +\infty} \angle\bm x(t) \bm y(t) = \frac{\pi}{2} \, .
\]
\end{lemma}

\begin{proof}
The map $t \mapsto \scalar{\bm x(t)}{\bm y(t)}/\bigl(\norm{\bm x(t)} \norm{\bm y(t)}\bigr)$ is definable in the
polynomially bounded structure $\mystruct$, and it is bounded by $1$. Therefore, it has a limit when $t \to +\infty$,
which proves the first part of the statement. Now, observe that:
\[
\val\Bigl(\frac{\scalar{\bm x}{\bm y}}{\norm{\bm x} \norm{\bm y}}\Bigr) \leq \max_{i \in [p]} (x_i + y_i) - \bigl(\max_{i \in [p]} x_i + \max_{i \in [p]} y_i\bigr) \, .
\]
As a consequence, if $\argmax_{i \in [p]} x_i \cap \argmax_{i \in [p]} y_i = \emptyset$, then the latter quantity is
negative. It follows that the ratio $\scalar{\bm x(t)}{\bm y(t)}/\bigl(\norm{\bm x(t)} \norm{\bm y(t)}\bigr)$ tends to
$0$ when $t \to +\infty$. This gives the claim.
\end{proof}

We will use Lemma~\ref{lemma:angle} in order to estimate the limit when $t \to +\infty$ of the angle between segments formed by successive points $\cpath_t(\lambda)$, $\cpath_t(\lambda')$ and $\cpath_t(\lambda'')$, \ie~with $\lambda < \lambda' < \lambda''$. One remarkable property is that the corresponding points of the tropical central path satisfy $\troppath(\lambda) \leq \troppath(\lambda') \leq \troppath(\lambda'')$. We refine Lemma~\ref{lemma:angle} to fit this setting:
\begin{lemma}\label{lemma:angle2}
Let $\bm U, \bm V, \bm W \in \K^p$, and $U := \val(\bm U)$, $V := \val(\bm V)$ and $W := \val(\bm W)$. If $\max_{i \in [p]} U_i < \max_{i \in [p]} V_i < \max_{i \in [p]} W_i$, and the sets $\argmax_{i \in [p]} V_i$ and $\argmax_{i \in [p]} W_i$ are disjoint, we have:
\[
\lim_{t \to +\infty} \angle\bm U(t) \bm V(t) \bm W(t) = \frac{\pi}{2} \, .
\]
\end{lemma}

\begin{proof}
Let us remark that for all $i \in [p]$, we have $\val(\bm V_i - \bm U_i) \leq \max(U_i, V_i)$, and this inequality is an equality if $U_i \neq V_i$. Since $\max_{i \in [p]} U_i < \max_{i \in [p]} V_i$, we deduce that $\max_{i \in [p]} \val(\bm V_i - \bm U_i) = \max_{i \in [p]} V_i$, and that the argument of the two maxima are equal. The same applies to the coordinates of the vector $\val(\bm W - \bm V)$. We deduce from Lemma~\ref{lemma:angle2} that $\angle\bm U(t) \bm V(t) \bm W(t)$ tends to $\pi / 2$ as soon as $\argmax_{i \in [p]} V_i \cap \argmax_{i \in [p]} W_i = \emptyset$.
\end{proof}

Given $U, V, W \in \trop^p$, this motivates us to introduce the \emph{tropical angle} $\angle^c U V W$ defined by $\angle^c U V W := \frac{\pi}{2}$ if $U, V, W$ satisfy the conditions of Lemma~\ref{lemma:angle2}, and $\angle^c U V W := 0$ otherwise.  We have the following result.
\begin{proposition}\label{prop:curvature}
Let $a, b \in \R$, and $\lambda_0 = a < \lambda_1 < \dots < \lambda_{q-1} < \lambda_q = b$. Then:
\[
\liminf_{t \to +\infty} \kappa\bigl(\cpath_t,[a, b]\bigr) \geq \sum_{k = 1}^{q-1} \angle^c \troppath(\lambda_{k-1}) \troppath(\lambda_k) \troppath(\lambda_{k+1}) \, .
\] 
\end{proposition}

\begin{proof}
This follows from Lemma~\ref{lemma:angle2} and from the fact that for all (sufficiently large) $t$, we have $
\kappa\bigl(\cpath_t,[a, b]\bigr) \geq \sum_{k = 1}^{q-1} \angle \cpath_t(\lambda_{k-1}) \cpath_t(\lambda_k) \cpath_t(\lambda_{k+1})$.
\end{proof}
Note that the previous result also holds for the primal and dual components of the central paths respectively. We are now ready to prove a lower bound on the curvature of the central path.

\begin{theorem}\label{th:curvature}
  The limit inferior of the curvature of the primal-dual (resp.\ primal) central path of the linear program
  $\hardyex_r(t)$ when $t \to +\infty$ is greater than or equal to $(2^{r-1}-1)\frac{\pi}{2}$.
\end{theorem}

\begin{proof}
  We use Proposition~\ref{prop:curvature} and provide a lower bound on the limit inferior of $\kappa(\cpath_t, [0,2])$
  by considering the subdivision of the range $[0,2]$ by the scalars $\lambda_k = \frac{4 k}{2^r}$ for $k = 0, \dots,
  2^{r-1}$.

  Let us first point out that, given $\lambda \in [0,2]$, all the dual components of the point $\troppath(\lambda)$
  of the tropical central path are less than or equal to $\max(0,\lambda - 1)$. This is a consequence of the
  identity~\eqref{eq:primal_dual} and the fact that all the primal components are greater than or equal to
  $\min(1,\lambda)$ by Proposition~\ref{prop:lift}. It follows that the dual components are dominated by the primal
  ones. In particular, using Table~\ref{tab:table}, we deduce that the maximal component of the vector
  $\troppath(\lambda_k)$ is equal to $r + \frac{2k+2}{2^r}$, and that is uniquely attained by the coordinate
  $z_r(\lambda)$ when $k$ is even, and by $z'_r(\lambda)$ when $k$ is odd. We immediately deduce that $\angle^c
  \troppath(\lambda_{k-1}) \troppath(\lambda_k) \troppath(\lambda_{k+1}) = \frac{\pi}{2}$, and we obtain the expected
  result using Proposition~\ref{prop:curvature}. The same proof provides the result for the primal central path.
\end{proof}

\begin{remark}\label{remark:curvature_dual}
We can refine Theorem~\ref{th:curvature} so as to include a lower bound in the curvature of the dual central path, up to considering a slightly modified version of $\hardyex_r$. More precisely, we introduce the extra variables $\bm u_{r+1}$ and $\bm v_{r+1}$, and add the inequalities $\bm u_{r+1} \leq \frac{1}{t^{r+1}} \bm u_r$ and $\bm v_{r+1} \leq \frac{1}{t^{r+1}} \bm v_r$ to the constraints defining $\hardyex_r$. After the introduction of slack variables, these inequalities become
\begin{align*}
\bm u_{r+1} + \bm z_{r+1} & = \frac{1}{t^{r+1}} \bm u_r \\
\bm v_{r+1} + \bm h_{r+1} & = \frac{1}{t^{r+1}} \bm v_r
\end{align*}
Using the technique involved in the proof of Proposition~\ref{prop:lift}, in particular the characterization of Corollary~\ref{cor:geometric_characterization}, we can easily show that the corresponding components of $u_{r+1}(\lambda)$, $v_{r+1}(\lambda)$, $z_{r+1}(\lambda)$ and $h_{r+1}(\lambda)$ of the primal tropical central path are given by the relations $u_{r+1}(\lambda) = z_{r+1}(\lambda) = -(r+1) + u_r(\lambda)$ and $v_{r+1} = h_{r+1} = -(r+1) + v_r(\lambda)$. The dual variables, respectively denoted by $u^d_{r+1}(\lambda)$, $v^d_{r+1}(\lambda)$, $z^d_{r+1}(\lambda)$ and $h^d_{r+1}(\lambda)$, are given by $u^d_{r+1}(\lambda) = z^d_{r+1}(\lambda) = \lambda - u_{r+1}(\lambda)$ and $v^d_{r+1}(\lambda) = h^d_{r+1}(\lambda) = \lambda - v_{r+1}(\lambda)$, thanks to~\eqref{eq:primal_dual}. Therefore, we obtain the values in Table~\ref{tab:othertable}.

\begin{table}[tb]
  \caption{Further central path coordinates for the modified Hardy linear program described in Remark~\ref{rem:modified}.  This extends Table~\ref{tab:table}.}
\label{tab:othertable}
\renewcommand{\arraystretch}{1.5}
\begin{tabular*}{.75\linewidth}{@{\extracolsep{\fill}}>{$}c<{$}>{$}c<{$}>{$}c<{$}>{$}c<{$}@{}}\toprule
\lambda & \frac{4k}{2^r} & \frac{4k+2}{2^r} & \frac{4k+4}{2^r} \\
\midrule
u_{r+1} & -1 + \frac{2k}{2^r} & -1 + \frac{2k+2}{2^r} & -1 + \frac{2k+2}{2^r}  \\
v_{r+1} & -1 + \frac{2k+1}{2^r} & -1 + \frac{2k+1}{2^r} & -1 + \frac{2k+3}{2^r}  \\
u^d_{r+1} & 1 + \frac{2k}{2^r} & 1 + \frac{2k}{2^r} & 1 + \frac{2k+2}{2^r}  \\
v^d_{r+1} & 1 + \frac{2k-1}{2^r} & 1 + \frac{2k+1}{2^r} & 1 + \frac{2k+1}{2^r} \\
\bottomrule
\end{tabular*}
\end{table}

Since the new primal and dual coordinates of the tropical central path are strictly less than the coordinates $z_r(\lambda)$ and $z'_r(\lambda)$ as soon as $r \geq 2$, the result of Theorem~\ref{th:curvature} remains unchanged. Concerning the dual central path, we note that at the points $\lambda'_k = 2 k / 2^r$ ($k \in \{0, \dots, 2^r\}$), the maximal components of the dual tropical central path are $u^d_{r+1}(\lambda'_k) = w^d_{r+1}(\lambda'_k)$ if $k$ is even, and $v^d_{r+1}(\lambda'_k) = h^d_{r+1}(\lambda'_k)$ if $k$ is odd, and they are equal to $1 + k / 2^r$. We deduce from Proposition~\ref{prop:curvature} that the limit inferior of the dual central path is greater than or equal to $(2^r - 1) \frac{\pi}{2}$.
\end{remark}

\begin{remark}
Picking $t$ of order $2^{\Omega(r^2 2^r)}$ is enough to obtain
a total curvature of $\Omega(2^r)$ for the central path of $\hardyex_r(t)$. 
This follows from a uniform estimate of the distance between the image of the classical central path under the $\log_t$ map and the tropical central path.
In general, such an estimate is provided by Theorem~\ref{th:uniform}. 
However, in the present case, an equivalent estimate is more
easily obtained by exploiting the simple structure of the constraint matrix.
We omit the proof as it relies on a routine but lengthy computation. 
\end{remark}

\bibliographystyle{alpha}

\end{document}